\documentclass[10pt, a4paper, english]{smfart}
\usepackage[english]{babel}
\usepackage{smfthm}
\usepackage[numbers]{natbib}
\usepackage{amssymb}
\usepackage{tikz-cd}
\usepackage{mathtools}
\usepackage{hyperref}
\usepackage{enumitem}
\usepackage{lipsum}
\renewcommand{\geq}{\geqslant}
\renewcommand{\leq}{\leqslant}

\theoremstyle{plain}
\newtheorem{theorem}[subsection]{Theorem}
\newtheorem{assumption}[subsection]{Assumption}
\newtheorem{definition}[subsection]{Definition}
\newtheorem{lemma}[subsection]{Lemma}
\newtheorem{conjecture}[subsection]{Conjecture}
\newtheorem{corollary}[subsection]{Corollary}
\newtheorem{proposition}[subsection]{Proposition}
\newtheorem{claim}[subsection]{Claim}

\theoremstyle{remark}
\newtheorem{remark}[subsection]{Remark}

\renewcommand{\AA}{\ensuremath{\mathbb{A}}}
\newcommand{\CC}{\ensuremath{\mathbb{C}}}

\newcommand{\PP}{\ensuremath{\mathbb{P}}}
\newcommand{\QQ}{\ensuremath{\mathbb{Q}}} 
 
\newcommand{\RR}{\ensuremath{\mathbb{R}}} 

\newcommand{\ZZ}{\ensuremath{\mathbb{Z}}}

\DeclareMathOperator{\GL}{GL}

\DeclareMathOperator{\pr}{pr}
\DeclareMathOperator{\Spec}{Spec}
\DeclareMathOperator{\End}{End}

\DeclareMathOperator{\Gal}{Gal}

\DeclareMathOperator{\MT}{MT}

\DeclareMathOperator{\CSpin}{CSpin}
\DeclareMathOperator{\Spin}{Spin}
\DeclareMathOperator{\SO}{SO}
\DeclareMathOperator{\Oo}{O}

\DeclareMathOperator{\Gmot}{G_{\mathrm{mot}}}
\DeclareMathOperator{\et}{\mathrm{\acute{e}t}}
 
\DeclareMathOperator{\NS}{\mathrm{NS}}

 \title{Galois representations on the cohomology of hyper-K\"{a}hler varieties}
 \author{Salvatore Floccari}
 \address{Institute of Algebraic Geometry, Leibniz University Hannover, Germany}
 \email{floccari@math.uni-hannover.de}
 
\begin{document}

 	
 	\maketitle
 	
 	\begin{abstract}
 	We show that the Andr\'{e} motive of a hyper-K\"{a}hler variety $X$ over a field $K \subset \CC$ with $b_2(X)>6$ is governed by its component in degree $2$. More precisely, we prove that if $X_1$ and $X_2$ are deformation equivalent hyper-K\"{a}hler varieties with $b_2(X_i)>6$ and if there exists a Hodge isometry $f\colon H^2(X_1,\QQ)\to H^2(X_2,\QQ)$, then the Andr\'e motives of $X_1$ and~$X_2$ are isomorphic after a finite extension of $K$, up to an additional technical assumption in presence of non-trivial odd cohomology. As a consequence, the Galois representations on the \'{e}tale cohomology of $X_1$ and~$X_2$ are isomorphic as well. 
 	We prove a similar result for varieties over a finite field which can be lifted to hyper-K\"{a}hler varieties for which the Mumford--Tate conjecture is true.
 	\end{abstract}

 	\section{Introduction}

A guiding principle in the study of hyper-K\"{a}hler manifolds is that many of their geometric properties are governed by their cohomology in degree $2$. Perhaps the most spectacular illustration of this principle is the global Torelli theorem due to Huybrechts, Markman and Verbitsky, which precisely explains to what extent the birational class of a hyper-K\"{a}hler manifold $X$ can be recovered 
from the integral Hodge structure on the lattice~$H^2(X,\ZZ)$. 

As another example, it is known that the total Hodge structure on the rational singular cohomology $H^*_X\coloneqq H^*(X,\QQ)$ is determined by the Hodge structure on $H^2_X$. This fact is a consequence of the properties of the Looijenga-Lunts-Verbitsky Lie algebra that was introduced in \cite{looijenga1997lie} and \cite{verbitsky1996cohomology}; a complete proof of this result has been given by Soldatenkov in \cite{solda19}.

Let now $K\subset \CC$ be a subfield which is finitely generated over $\QQ$ and let $X$ be a hyper-K\"{a}hler variety over $K$. We fix a prime number $\ell$ and consider the \'{e}tale cohomology groups $H^*_{X,\ell}\coloneqq H^*_{\et}(X_{\bar{K}},\QQ_{\ell})$ of $X$. It is then natural to ask whether the Galois representation on the $\ell$-adic cohomology $H^*_{X,\ell}$ of $X$ is determined by its restriction to $H^2_{X,\ell}$. Going even further, we may ask the analogous question at the level of homological or Chow motives - however, in this setting the existence of K\"{u}nneth components is not known in general.
To circumvent this difficulty, we will work with Andr\'{e}'s category of motives \cite{andre1996Motives}.

We prove that, up to a finite field extension of the base field, the total Andr\'{e} motive $\mathcal{H}^*(X)=\bigoplus_j \mathcal{H}^{j}(X)$ of a hyper-K\"{a}hler variety $X$ is governed by its component in degree $2$. In what follows we say that two hyper-K\"{a}hler varieties $X_1$ and~$X_2$ over a field $K\subset \CC$ are deformation equivalent if the complex varieties $X_{1,\CC}$ and $X_{2,\CC}$ are deformation equivalent in the analytic sense. We let $H^*_{X_i}\coloneqq \bigoplus_j H^j(X_{i, \CC},\QQ)$.

\begin{theorem}[= Theorem \ref{thm:mot1}]\label{thm:mot}
	Let $X_1,X_2$ be deformation equivalent hyper-K\"{a}hler varieties with $b_2(X_i)>6$ over a field $K\subset \CC$. If $X_i$ has non-trivial cohomology in odd degree, assume that, for $i=1,2$, the motive $\mathcal{H}^1_{A_i}$ belongs to the tannakian category generated by $\mathcal{H}_{X_i}^*$, where $A_i$ is the Kuga-Satake variety on $H^2(X_i)$.
	Assume that there exists a Hodge isometry $f\colon H^2_{X_1}\to H^2_{X_2}$. 
	 Then, there exist a finite field extension $K'/K$ and an isomorphism of graded algebras $F\colon H^*_{X_1}\to H^*_{X_2}$ which is the realization of an isomorphism of Andr\'{e} motives $\mathcal{H}^*_{X_{1, K'}}\to \mathcal{H}^*_{X_{2, K'}}$ over~$K'$.
\end{theorem}

The map $F$ obtained in the theorem is in particular an isomorphism of Hodge structures. The assumption on the second Betti number is needed to show that deformation equivalent hyper-K\"{a}hler varieties can be connected via a chain of polarized deformations; see Theorem \ref{thm:connectingHK} for the precise statement.
All known hyper-K\"{a}hler varieties have $b_2\geq 7$.

In presence of non-trivial odd cohomology we need an extra assumption to control the odd part of the motive. 
The Kuga-Satake variety~$A_i$ is an abelian variety closely related to $H^2(X_i)$ (\cite{deligne1971conjecture}); we expect that its motive belongs to the tannakian category generated by $\mathcal{H}_{X_i}^*$. By \cite{FFZ} this happens for the known hyper-K\"{a}hler varieties with non-trivial odd cohomology, which are those of generalized Kummer deformation type. We obtain the following consequence of Theorem~\ref{thm:mot}.

\begin{corollary}[{=Corollary \ref{cor:unconditionalGalRep1}}]\label{cor:unconditionalGalRep}
	Let $K\subset \CC$ be a subfield which is finitely generated over $\QQ$ and let $X_1,X_2$ be deformation equivalent hyper-K\"{a}hler varieties over~$K$ such that $b_2(X_i)>6$. If $X_i$ has non-trivial cohomology in odd degree, assume further that, for $i=1,2$, the motive $\mathcal{H}^1_{A_i}$ belongs to the tannakian category generated by $\mathcal{H}_{X_i}^*$, where $A_i$ is the Kuga-Satake variety on $H^2(X_i)$. Assume that there exists a $\Gal(\bar{K}/K)$-equivariant isometry $f\colon H^2_{X_1, \ell}\to H^2_{X_2,\ell}$. 
	Then, there exist a finite field extension $K'/K$ and a $\Gal(\bar{K}/K')$-equivariant isomorphism of graded algebras $F\colon H^*_{X_1, \ell}\to H^*_{X_2,\ell}$.
\end{corollary}

The corollary is the $\ell$-adic counterpart of the Hodge theoretic result from \cite{solda19}. This is not surprising: if the Mumford--Tate conjecture was true, the corollary would be a direct consequence of its analogue in Hodge theory. Even though the Mumford--Tate conjecture is not known for arbitrary hyper-K\"{a}hler varieties, in \cite{FFZ} we proved it for all hyper-K\"{a}hler varieties of known deformation type. 
When the conjecture holds, we obtain a more precise result on the Galois representations $H^*_{X_{i,\ell}}$. Let us say that two hyper-K\"{a}hler varieties $X_1$ and $X_2$ are $H_{\ell}^*$-equivalent if there exists an isomorphism of graded algebras $H^*_{X_1,\ell}\cong H^*_{X_2,\ell}$ which restricts to an isometry in degree $2$ with respect to the Beauville-Bogomolov pairings. Deformation equivalent $X_1$ and $X_2$ are $H_{\ell}^*$-equivalent as well, since in this case $X_{1,\CC}$ and $X_{2,\CC}$ are homeomorphic. 

\begin{theorem}[{= Theorem \ref{thm:main}}]\label{thm:GalRepHK}
	Let $K_1, K_2$ be subfields of $\CC$, finitely generated over $\QQ$, and let $X_i$ be a hyper-K\"{a}hler variety over $K_i$, for $i=1,2$. Assume that~$X_1$ and $X_2$ are $H_{\ell}^*$-equivalent, and that the Mumford--Tate conjecture holds for both of them. Let $\Gamma\subset \Gal(\bar{K}_1/K_1)$ be a subgroup, let $\epsilon\colon \Gamma\to \Gal(\bar{K}_2/K_2)$ be a homomorphism; $\Gamma$ acts on $H^*_{X_1,\ell}$ via its inclusion into $\Gal(\bar{K}_1/K_1)$ and on $H^*_{X_2,\ell}$ via $\epsilon$. Assume that there exists a $\Gamma$-equivariant isometry $f\colon H^2_{X_1, \ell} \to H^2_{X_2, \ell}$.
	Then, there exist a subgroup $\Gamma'\subset \Gamma$ of finite index and an isomorphism $F\colon H^*_{X_1, \ell}\to H^*_{X_2, \ell}$ of graded algebras which is $\Gamma'$-equivariant.
\end{theorem}

This result leads to similar conclusions for hyper-K\"{a}hler varieties over finite fields. The study of such varieties is still in its early stages: besides $\mathrm{K}3$ surfaces, it is not clear how to define these objects (but a possible definition is proposed in \cite{fuLi2018}). Despite this, certain higher dimensional moduli spaces of sheaves on $\mathrm{K3}$ surfaces play a key role in Charles's proof of the Tate conjecture for $\mathrm{K}3$ surfaces \cite{charles2016birational} over finite fields; other examples of hyper-K\"{a}hler varieties over finite fields can be obtained from moduli spaces of sheaves on abelian surfaces~\cite{fuLi2018}. 
All these varieties can be lifted to some hyper-K\"{a}hler variety in characteristic $0$. We will therefore consider smooth and projective varieties over finite fields which can be lifted to a hyper-K\"{a}hler variety in characteristic $0$; see \S\ref{subsec:setting} for the precise meaning of this. The recent article \cite{yang2019irreducible} shows that this approach yields at least a good notion of varieties of $\mathrm{K3}^{[n]}$-type.

Let $k$ be a finite field and let $\ell$ be a prime number invertible in $k$. For a smooth and projective variety $Z$ over $k$, we let $H^*_{Z,\ell}\coloneqq H^*_{\et}(Z_{\bar{k}}, \QQ_{\ell})$; if $Z$ can be lifted to a hyper-K\"{a}hler variety in characteristic $0$ then the second cohomology group~$H^2_{Z,\ell}$ inherits a non-degenerate $\QQ_{\ell}$-valued symmetric bilinear pairing, see Remark \ref{rem:BBform}.

\begin{theorem}[{= Theorem \ref{thm:Galfinitefields}}]\label{thm:GalRep2}
Let $Z_1, Z_2$ be smooth projective varieties over~$k$ such that there exist $H_{\ell}^*$-equivalent hyper-K\"{a}hler varieties $X_1$, $X_2$ in characteristic $0$ which lift $Z_1$ and $Z_2$ respectively. Assume that the Mumford--Tate conjecture holds for both $X_1$ and $X_2$ and that there exists a $\Gal(\bar{k}/k)$-equivariant isometry $f\colon H^2_{Z_1, \ell} \to H^2_{Z_2,\ell}$. Then, there exist a finite field extension $k'$ of $k$ and a $\Gal(\bar{k}/k')$-equivariant isomorphism of graded algebras $F\colon H^*_{Z_1 ,\ell}\to H^*_{Z_2,\ell}$.
\end{theorem} 

In particular $Z_{1, k'}$ and $Z_{2,k'}$ have the same zeta function. 
In the special case when $Z_1$ and $Z_2$ are moduli spaces of stable sheaves on $\mathrm{K}3$ surfaces over $k$ the above statement has already been proven by Frei in \cite{frei2018} via a different method, which uses Markman's results from \cite{markman2008}. 

The structure of this article is reversed with respect to the order of the introduction. Namely, after reviewing in \S\ref{sec:LLV} the construction of the Looijenga-Lunts-Verbitsky (LLV) Lie algebra, we use the properties of this Lie algebra to prove Theorems \ref{thm:GalRepHK} and \ref{thm:GalRep2} in~\S\ref{sec:GalRep}. 
We then prove Theorem \ref{thm:mot} and Corollary \ref{cor:unconditionalGalRep} in~\S\ref{sec:motivic}. These last results rely fundamentally on the \textit{defect groups} of hyper-K\"{a}hler varieties introduced in \cite{FFZ} with Lie Fu and Ziyu Zhang. The proof of Theorem~\ref{thm:mot} also uses the fact that deformation equivalent hyper-K\"{a}hler varieties can be connected using polarized deformations, which is proven in the last section \S\ref{sec:ProjFamilies}.

 \subsection*{Notation and conventions}
 
 By a hyper-K\"{a}hler variety over a field $K \subset \CC$ we mean a smooth and projective variety $X$ over $K$ such that $X(\CC)$ is a complex hyper-K\"{a}hler manifold, i.e.\ it is simply connected and $H^0(X(\CC),\Omega^2)$ is spanned by the cohomology class of a nowhere degenerate holomorphic $2$-form. If $X_1,X_2$ are hyper-K\"{a}hler varieties over subfields $K_1,K_2\subset  \CC$ respectively, we say that $X_1$ and $X_2$ are deformation equivalent if the complex manifolds $X_1(\CC)$ and $X_2(\CC)$ are deformation equivalent (in the analytic sense). For a smooth and projective variety $X$ over a subfield $K\subset \CC$ we use the notation $H^j_X\coloneqq H^j(X(\CC), \QQ)$ and $H^j_{X,\ell}\coloneqq H^j_{\et}(X_{\bar{K}}, \QQ_{\ell})$, where $\bar{K}\subset \CC$ is the algebraic closure of $K$ in $\CC$ and $\ell$ is a prime number.

 \subsection*{Aknowledgements}
 
 I am most grateful to Ben Moonen and Arne Smeets for their help and encouragement. I wish to thank Lie Fu for many useful discussions around the topics of this paper. I also thank the referee for his/her careful review.

 \section{The LLV-Lie algebra}\label{sec:LLV}
 
 \subsection{} 
 In this section, we let $X$ be a complex hyper-K\"{a}hler variety of dimension~$2n$. We let $H^*_{X}\coloneqq \bigoplus_j H^{j}(X,\QQ)$; the second cohomology group $H^2_X$ is equipped with the Beauville-Bogomolov pairing $q$, a non-degenerate symmetric bilinear form, see \cite[Th\'eor\`eme 5]{beauville1983varietes}.  We define the Mukai extension of the quadratic space $(H^2_{X}, q)$ as the vector space
 $
 \tilde{H}^2_X \coloneqq \QQ \cdot v \, \oplus \, {H}^2_{X} \, \oplus \, \QQ\cdot w $, equipped with the pairing 
 \[
 \tilde{q} \bigl( (av,b,c w), (a'v, b', c' w)\bigr)= q(b,b')- a c' - a'c .
 \]
 
 \subsection{} 
Given $x\in H^2_{X}$, let $L_x\colon H^*_{X} \to H_{X}^{*+2}$ be given by cup-product with $x$. We say that $x$ has the Lefschetz property if the maps 
 $L_x^k\colon H^{2n-k}_{X} \to H^{2n+k}_{X}$ 
 are isomorphisms for all $k>0$. Let $\theta$ denote the endomorphism of the cohomology which acts on~$H^j_{X}$ as multiplication by $j-2n$. It is well-known that the class $x$ has the Lefschetz property if and only if there exists $\Lambda_x\colon H_{X}^* \to H_{X}^{*-2}$ such that $(L_x, \theta, \Lambda_x)$ is an $\mathfrak{sl}_2$-triple, meaning that we have 
 \[
 [\theta,L_x]=2L_x, \ [\theta, \Lambda_x]=-2\Lambda_x, \ [L_x,\Lambda_x]=\theta.
 \] 
 If it exists, $\Lambda_x$ is uniquely determined. The subset of $x\in H^2_{X}$ with the Lefschetz property is Zariski open in $H^2_X$, and the first Chern class of an ample divisor on~$X$ has the Lefschetz property by the Hard Lefschetz theorem. 
 \begin{definition} 
 	The LLV-Lie algebra $\mathfrak{g}(X)$ of $X$ is the Lie subalgebra of $\mathfrak{gl}(H^*_{X}) $ generated by all $\mathfrak{sl}_2$-triples $(L_x, \theta, \Lambda_x)$ for $x\in H^2_{X}$ with the Lefschetz property. We let $\mathfrak{g}_0(X)\subset \mathfrak{g}(X)$ denote the centralizer of the semisimple element $\theta$.
 \end{definition}
In other words, $\mathfrak{g}_0(X)$ consists of those endomorphisms in $\mathfrak{g}(X)$ whose action on $H^*_X$ preserve the grading.
 
 The LLV-Lie algebras of hyper-K\"{a}hler varieties have been fully described.
 
 \begin{theorem}[{\cite{verbitsky1996cohomology}, \cite{looijenga1997lie}}] 
 	\phantomsection
 	\label{thm:LLV}
 	\begin{enumerate}[label={(\alph*)}]
 		\item There exists a unique isomorphism of $\QQ$-Lie algebras 
 		\[
 		\varphi \colon \mathfrak{g}(X)\xrightarrow{\ \sim \ } \mathfrak{so} (\tilde{H}^2_{X},\tilde{q}) ,
 		\]
 		such that:
 		\begin{itemize}
 			\item $\varphi(\theta)$ vanishes on $H^2_{X}$, $\varphi(\theta)(v)=-2v$ and $\varphi(\theta)(w)=2w$, and
 			\item for any $x\in H^2_{X}$ with the Lefschetz property, we have $\varphi(L_{x})(v)= x$, $\varphi(L_x)(w)=0$ and $\varphi(L_x)(y)= q(x,y) \cdot w$, for all $y\in H^2_{X}$.
 		\end{itemize} 
 		\item The isomorphism $\varphi$ restricts to an isomorphism 
 		\[\mathfrak{g}_0(X)\cong \mathfrak{so} (H^2_{X}, q)\oplus \QQ \cdot \varphi(\theta);\] the induced representation of $ \mathfrak{so}(H^2_{X}, q\bigr)$ on $H^2_{X}$ is the standard representation.
 	\end{enumerate}
 \end{theorem}

For later use, we note the following functoriality property of the LLV-construction.
\begin{lemma}\label{lem:functoriality}
	Let $X_1$ and $X_2$ be hyper-K\"{a}hler varieties and let $F\colon H^*_{X_1} \to H^*_{X_2}$ be an isomorphism of graded algebras. Then the induced isomorphism $F_*\colon \GL(H^*_{X_1})\to \GL(H^*_{X_2})$ given by $A\mapsto FAF^{-1}$ restricts to an isomorphism $\mathfrak{g}(X_1)\to \mathfrak{g}(X_2)$.
\end{lemma}
\begin{proof}
	Let $x\in H^2_{X_1}$ be an element with the Lefschetz property, and consider the corresponding $\mathfrak{sl}_2$-triple $(L_x, \theta_{X_1}, \Lambda_x)$. Then $(FL_xF^{-1}, F\theta_{X_1} F^{-1}, F\Lambda_x F^{-1})$ is again an $\mathfrak{sl}_2$-triple; moreover, since $F$ is an isomorphism of graded algebras it is immediate to check that $FL_xF^{-1}=L_{F(x)}$ and $F\theta_{X_1} F^{-1}=\theta_{X_2}$. Further, $F(x)$ has the Lefschetz property as well, and it follows that $F_*(\Lambda_x)=\Lambda_{F(x)}$. Since the Lie algebra $\mathfrak{g}(X_1)$ is generated by the $\mathfrak{sl}_2$-triples $(L_x, \theta_{X_1}, \Lambda_x)$ as above, this concludes the proof.
\end{proof}
\begin{remark}\label{rem:fujiki1}
	Together with Theorem \ref{thm:LLV}, Lemma \ref{lem:functoriality} implies that any linear isomorphism $f\colon H^2_{X_1}\to H^2_{X_2}$ which extends to an isomorphism $H^*_{X_1}\to H^*_{X_2}$ of graded algebras induces an isomorphism $f_*\colon \mathfrak{so}(H^2_{X_1})\to \mathfrak{so}(H^2_{X_2})$. This fact can be seen as a consequence of Fujiki's relation, which gives positive rational constants $\lambda_i$, for $i=1,2$, such that $q(\alpha,\alpha)^{n}=\lambda_i \int_{X_i} \alpha^{2n}$ for any $\alpha\in H^2_{X_i}$, see \cite[\S1.11]{Huy99}.
\end{remark}

 \subsection{} 
We let $\mathrm{G}(X)$ be the semisimple simply connected algebraic group with Lie algebra $\mathfrak{g}(X)$, and let $\mathrm{G}_0(X)\subset \mathrm{G}(X)$ be the unique connected subgroup with Lie algebra $\mathfrak{g}_0(X)$. By Theorem \ref{thm:LLV} we have an isomorphism
\[
\tilde{\varphi}\colon \mathrm{G}(X) \xrightarrow{\ \sim \ } \Spin(\tilde{H}^2_X, \tilde{q}).
\]
Let $U\coloneqq \langle v,w \rangle$, equipped with the restriction of $\tilde{q}$. Since $\tilde{H}^2_X= H^2_X\oplus U$, we can view $\Spin(H^2_X, q)$ and $\Spin(U)$ as algebraic subgroups of $\Spin( \tilde{H}^2_X, \tilde{q} )$. We have $\Spin(U)\cong \mathbb{G}_m$, and the Lie algebra of $\Spin(U)\subset \Spin(\tilde{H}^2_X,\tilde{q})$ is $\QQ\cdot \varphi(\theta)$. Moreover, $\Spin(H^2_X, q)\cap \Spin(U)=\{\pm 1\} $. We conclude that $\tilde{\varphi}$ restricts to an isomorphism
\[
\tilde{\varphi}\colon \mathrm{G}_0(X)\xrightarrow{\ \sim \ } \CSpin(H^2_X, q)=\Spin(H^2_X, q)\cdot \Spin(U).
\]

The above assertions are checked as follows. With respect to the basis $\{v, -\frac{w}{2} \}$, the matrix of $\tilde{q}_{|_U} $ is $\Bigl(\begin{smallmatrix}
0 & 1/2 \\
{1}/{2} & 0
\end{smallmatrix}\Bigr)$. Let $\mathrm{Cl}(U)$ be the Clifford algebra on $U$. Then $\mathrm{Cl}(U)$ is identified with the algebra of $2$ by $2$ matrices with coefficients in $\QQ$; an isomorphism is given by 
\[
v\mapsto \begin{pmatrix} 0 & 0 \\ 1 & 0 \end{pmatrix}, \ -\frac{w}{2} \mapsto \begin{pmatrix} 0 & 1 \\ 0 & 0 \end{pmatrix}.
\]
The even Clifford algebra $\mathrm{Cl}^+(U)$ consists of the diagonal matrices, while $\mathrm{Cl}^{-}(U)$ consists of those matrices with $0$ on the diagonal. The spinor norm $\mathrm{Cl}(U)^{\times }\to \QQ^{\times} $ is the determinant. Therefore $\Spin(U)\cong \mathbb{G}_{m}$ is the standard maximal torus of $\mathrm{SL}_2$. The adjoint action of $\Spin(U)$ on $\tilde{H}^2_X$ is trivial on the summand $H^2_{X}$, and we have
\[
\begin{pmatrix} \lambda & 0 \\ 0 & \lambda^{-1} \end{pmatrix} v \begin{pmatrix} \lambda^{-1} & 0 \\ 0 & \lambda \end{pmatrix}  = \lambda^{-2} v, \ \begin{pmatrix} \lambda & 0 \\ 0 & \lambda^{-1} \end{pmatrix}  w \begin{pmatrix} \lambda^{-1} & 0 \\ 0 & \lambda \end{pmatrix}  = \lambda^{2} w.
\]
 This implies that the subgroup $\Spin(U)\subset \Spin(\tilde{H}^2_X, \tilde{q})$ corresponds to the Lie subalgebra $\QQ\cdot \theta$ of $\mathfrak{so}(\tilde{H}^2_{X}, \tilde{q})$. Finally, since $\mathrm{Cl}(\tilde{H}^2_{X}, \tilde{q})= \mathrm{Cl}(H^2_{X}, q)\otimes \mathrm{Cl}(U, \tilde{q}_{|_U})$, it is clear that we have $\Spin(H^2_{X}, q)\cap \Spin(U)=\{\pm 1\}$.
 
\subsection{}
The action of $\mathfrak{g}(X)$ on $H^*_X$ integrates to a representation $\tilde{\rho}\colon \mathrm{G}(X)\to \GL(H^*_X)$, which restricts to 
\[
\tilde{\rho}_0\colon \mathrm{G}_0(X) \to \prod_j \GL(H^j_X).
\]
We denote by $\tilde{\rho}_0^{(2)}\colon \mathrm{G}_0(X)\to \GL(H^2_X)$ its degree $2$ component.

In what follows, we identify $\mathrm{G}(X)$ with $\Spin(\tilde{H}^2_X, \tilde{q})$ and $\mathrm{G}_0(X)$ with $\CSpin(H^2_X, q)$ via $\tilde{\varphi}$. If there is no risk of confusion, we simply write $ \CSpin(H^2_X)$ for $\CSpin(H^2_X, q)$, and similarly for other groups.

\begin{remark} 
\label{rem:minusone}
Consider the element $-1\in \CSpin(H^2_X)\subset \Spin(\tilde{H}^2_X)$. It  has been shown by Verbitsky in \cite[\S8]{verbitsky1995mirror} that for all $j\geq 0$ and any $v\in H^j_X$ we have $\tilde{\rho}(-1) (v)=(-1)^j$. Combining this with Theorem \ref{thm:LLV}, it follows that $\tilde{\rho}$ is faithful if~$X$ has non-trivial cohomology in some odd degree, and that $\tilde{\rho}$ has kernel $\{\pm 1\}$ otherwise.
\end{remark}
\subsection{} 
The connected center of the algebraic group $\CSpin(H^2_X)$ is the subgroup $\mathbb{G}_m$ of invertible scalars in the Clifford algebra, and we have short exact sequences of algebraic groups
\[
\begin{tikzcd} 
1 \arrow{r} & \mathbb{G}_m \arrow{r} & \CSpin(H^2_X) \arrow{r}{\pi} &  \SO(H^2_X) \arrow{r} & 1
\end{tikzcd}
\]
and
\[
\begin{tikzcd} 
1 \arrow{r} & \Spin(H^2_X) \arrow{r} & \CSpin(H^2_X) \arrow{r}{\mathrm{Nm}} &  \mathbb{G}_m \arrow{r} & 1 
\end{tikzcd}
\]
such that for all $z\in \mathbb{G}_m\subset \CSpin(H^2_X)$ we have $\mathrm{Nm}(z)=z^{2}$.

In addition to the representation $\tilde{\rho}_0$, we will consider a second, twisted, action of $\CSpin(H^2_X)$ on $H^*_X$, which we will refer to as the $R$-action. It is defined via the homomorphism
\[
R\colon \CSpin(H^2_X) \to \prod_j \GL(H^j_X)
\]
given by $R(g)=\mathrm{Nm}(g)^{n} \cdot \tilde{\rho}_0(g)$.

\begin{lemma}\label{lem:byauto}
	The $R$-action on $H^*_X$ is an action by graded algebra automorphisms.
\end{lemma}
\begin{proof}
	It has been shown in \cite[(4.4)]{looijenga1997lie} that the semisimple part $\mathfrak{so}(H^2_X)$ of the Lie algebra of $\CSpin(H^2_X)$ acts on the cohomology algebra via derivations; it follows that the subgroup $\Spin(H^2_X)$ acts on $H^*_X$ by graded algebra isomorphisms. 
	Moreover for any $z\in \mathbb{G}_m$ and $y\in H^j_X$ we have $\tilde{\rho}_0(z)(y)=z^{j-2n}\cdot y$. Thus the factor $\mathrm{Nm}(z)^{n}= z^{2n}$ ensures that also the action of $\mathbb{G}_m\subset \CSpin(H^2_X)$ on $H^*_X$ is by algebra automorphisms. As $\CSpin(H^2_X)=\mathbb{G}_m\cdot \Spin(H^2_X)$, this concludes the proof.
\end{proof}

\begin{remark}\label{rem:isogenies}
The homomorphism
\[
(\mathrm{Nm}, \pi)\colon \CSpin(H^2_X) \to \mathbb{G}_m\times \SO(H^2_X)
\]
is surjective with kernel $\{\pm 1 \}$. By Remark {\ref{rem:minusone}}, the $R$-action on the even cohomology factors through $(\mathrm{Nm}, \pi)$. If $g\in \CSpin(H^2_X)$, then the degree $2$ component $R^{(2)}(g)$ of $R(g)$ equals $\mathrm{Nm}(g) \cdot \pi(g) $, while for $\tilde{\rho}_0(g)$ we have $\tilde{\rho}^{(2)}_0(g)= \mathrm{Nm}(g)^{1-n}\cdot \pi(g)$.

The combination of this observation with Theorem \ref{thm:LLV} implies that the natural homomorphism $R\bigl(\mathrm{G}_0(X)\bigr) \to R^{(2)}\bigl(\mathrm{G}_0(X)\bigr)$ is an isomorphism if the odd cohomology of $X$ vanishes, and it has kernel $\{\pm 1\}$ otherwise.
\end{remark} 
 
 \subsection{}
 Given a $\QQ$-Hodge structure $V$, we let $\MT(V)$ denote its Mumford--Tate group. 
 As a consequence of a result of Verbitsky \cite{verbitsky1996cohomology}, it is known that the LLV-Lie algebra of a hyper-K\"{a}hler variety controls the Hodge structure on its cohomology. We summarize this result as follows, see {\cite[Lemma 6.7]{FFZ}}.
 
 \begin{theorem} \label{thm:HodgeHK}
 	 The Mumford--Tate group $\MT(H^*_X)$ is contained in the image of the representation $R\colon \CSpin(H^2_X)\to \prod_j \GL(H^j_X)$.
\end{theorem}

We have the following consequence.

\begin{proposition}\label{prop:hodge}
Let $X_1$ and $X_2$ be complex hyper-K\"{a}hler varieties and let $F\colon H^*_{X_1}\to H^*_{X_2}$ be an isomorphism of graded algebras. Assume that the degree $2$ component $F^{(2)}\colon H^2_{X_1}\to H^2_{X_2}$ is an isomorphism of Hodge structures. Then $F$ is an isomorphism of Hodge structures.
\end{proposition}

\begin{proof} 
Let $\mathbb{S}\coloneqq \mathrm{Res}^{\CC}_{\RR}(\mathbb{G}_m)$ be the Deligne torus. The total Hodge structure on~$H^*_{X_i}$ corresponds to a real representation
\[
h_{X_i}\colon \mathbb{S}\to \prod_j \GL(H^j_{X_i}) \otimes \RR;
\]
by definition, $h_{X_i}$ factors through $\MT(H^*_{X_i})(\RR)$.
By Theorem \ref{thm:HodgeHK} the group $\MT(H^*_{X_i})$ is contained in the image of the representation $R\colon \mathrm{G}_0(X_i) \to \prod_j \GL(H^j_{X_i})$. 
By Lemma \ref{lem:functoriality}, the induced isomorphism $F_*\colon \GL(H^*_{X_1})\to \GL(H^*_{X_2})$ restricts to an isomorphism
$\mathfrak{g}(X_1)\cong \mathfrak{g}(X_2)$; since moreover $F$ preserves the cohomological grading, $F_*$ restricts to an isomorphism 
$R\bigl(\mathrm{G}_0(X_1)\bigr)\cong R\bigl({\mathrm{G}}_0(X_2)\bigr)$.

We have to prove that the diagram
\[
\begin{tikzcd}
& \mathbb{S} \arrow[swap]{ld}{h_{X_1}} \arrow{rd}{h_{X_2}} \\ 
R\bigl({\mathrm{G}}_0(X_1)\bigr) (\RR) \arrow{d}{\mathrm{pr}_1}  \arrow{rr}{F_*}  &&  R\bigl({\mathrm{G}}_0(X_2)\bigr)(\RR)  \arrow{d}{\mathrm{pr}_2} \\
R^{(2)}\bigl({\mathrm{G}}_0(X_1)\bigr)({\RR})  \arrow{rr}{F^{(2)}_*}  && R^{(2)}\bigl({\mathrm{G}}_0(X_2)\bigr)({\RR}) 
\end{tikzcd}
\]
is commutative. By Remark \ref{rem:isogenies}, the morphism $\mathrm{pr}_2\colon R\bigl({\mathrm{G}}_0(X_2)\bigr)\to R^{(2)}\bigl({\mathrm{G}}_0(X_2)\bigr)$ is either an isomorphism or a central isogeny of degree $2$; let $C$ be the kernel.
Since $F^{(2)}$ is an isomorphism of Hodge structures, we have $F_*^{(2)}\circ \mathrm{pr}_1 \circ h_{X_1} = \mathrm{pr}_2\circ h_{X_2}$. Hence, there is a morphism $\xi\colon \mathbb{S} \to C$ such that $F_*\circ h_{X_1} = \xi \cdot h_{X_2}$. But $\mathbb{S}$ is connected and $C$ is finite, so $\xi$ is trivial and $F$ is an isomorphism of Hodge structures.
\end{proof} 
\begin{remark}
	In fact Theorem \ref{thm:HodgeHK} (and hence also the Proposition) holds more generally for complex hyper-K\"ahler manifolds which are not necessarily projective. 
\end{remark}

\section{Galois representations}\label{sec:GalRep}

\subsection{} 
Throughout this section, $\ell$ will denote a fixed prime number. Let $K\subset \CC$ be a field that is finitely generated over $\QQ$, and let $\bar{K}$ be the algebraic closure of $K$ in~$\CC$. By a hyper-K\"{a}hler variety over $K$ we mean a smooth projective variety $X$ over $K$ such that the base change $X_{\mathbb{C}}$ is a complex hyper-K\"{a}hler variety. 
For all integers~$j$ and $m$, we have canonical comparison isomorphisms
\[
H_{\et}^{j}(X_{\bar{K}}, \QQ_{\ell})(m) \cong H^{j}(X_{\CC}, \QQ)(m) \otimes_{\QQ} \QQ_{\ell} .
\]
We let $H^*_{X}\coloneqq  \bigoplus_j H^{j}(X_{\CC}, \QQ)$ and adopt the notation from the previous section. We define $H^{*}_{X, \ell} \coloneqq \bigoplus_j H_{\et}^{j}(X_{\bar{K}}, \QQ_{\ell})$. We will identify $H^j_{X, \ell}$ with $H^j_{X}\otimes_{\QQ} \QQ_{\ell}$; the Beauville-Bogomolov form extends to a non-degenerate $\QQ_{\ell}$-valued bilinear pairing on $H^2_{X,{\ell}}$.

There is a continuous representations 
\[
\sigma_X\colon \Gal(\bar{K}/K) \to \prod_j \GL(H^j_{X,\ell}).
\]
We denote by $G(H^*_{X,\ell})\subset \GL(H^*_{X,\ell})$ the Zariski closure of the image of $\sigma_X$. 
The group $G(H^*_{X,\ell})$ is not connected in general. The identity component $G(H^*_{X,\ell})^0$ of this group remains invariant under finitely generated field extensions of $K$. After replacing $K$ with a finite field extension $\hat{K}/K$, the group $G(H^*_{X,\ell})$ becomes connected, see~\cite[Remarks 2.2.2]{moonen17}.

\subsection{} 
Let $K\subset \bar{K}\subset \CC$ be as above. The Mumford--Tate conjecture aims to compare the Hodge structure on $H^*_{X}$ with the Galois representation $\sigma_X$ on $H^*_{X,\ell}$.

\begin{conjecture}[Mumford--Tate conjecture] \label{conj:mtc}
	For any smooth and projective variety $X$ over $K$, the comparison isomorphism $H^{*}_{X,\ell}\cong H^*_{X}\otimes \QQ_{\ell}$ induces an isomorphism of connected algebraic groups
	\[
	G(H^*_{X,\ell})^{0} \cong \MT(H^*_{X}) \otimes_{\QQ} \QQ_{\ell}.
	\]
\end{conjecture} 
Let us note that the version of the Mumford--Tate conjecture given here is stronger than the one which says that under the comparison isomorphism $H^j_{X,\ell}\cong H^{j}_{X}\otimes_{\QQ} \QQ_{\ell}$ the group $G(H^j_{X,\ell})^{0}$ is identified with $\MT(H^j_{X}) \otimes \QQ_{\ell}$ for all $j$.

The Mumford--Tate conjecture does not depend on the base field, and it may even be formulated for varieties over the complex numbers, see \cite[\S1.6]{moonen2017}.

\subsection{} 
At present, four deformation types of complex hyper-K\"{a}hler varieties are known besides $\mathrm{K}3$ surfaces, commonly referred to as the deformation types $\mathrm{K}3^{[n]}$ and $\mathrm{Kum}^n$ (\cite{beauville1983varietes}), for all $n\geq 2$, and O'Grady's deformation types $\mathrm{OG}10 $ (\cite{O'G99}) and $\mathrm{OG}6$ (\cite{O'G03}). Together with Lie Fu and Ziyu Zhang, we have proven the Mumford--Tate conjecture for all hyper-K\"{a}hler varieties of one of these types (the case of $\mathrm{K}3^{[n]}$-type varieties was dealt with in \cite{floccari2019}).

\begin{theorem}[{\cite[Theorem 1.18]{FFZ}}]\label{thm:MTCHK}
	Let $X$ be a hyper-K\"{a}hler variety over~$K$. Assume that $X$ is deformation equivalent to one of the known examples, that is, $X_{\CC}$ is of deformation type $\mathrm{K}3^{[n]}$, $\mathrm{Kum}^n$, $\mathrm{OG}10 $ or $\mathrm{OG}6$. Then the Mumford--Tate Conjecture \ref{conj:mtc} holds for $X$.
\end{theorem}  

The Mumford--Tate conjecture in degree $2$ is known for arbitrary hyper-K\"{a}hler varieties $X$ with $b_2(X)>3$ by \cite{Andre1996}; see also \cite{moonen2017}.

\subsection{} 

Let now $K_1$, $K_2$ be subfields of $\CC$, finitely generated over $\QQ$, and consider hyper-K\"{a}hler varieties $X_ 1$, $X_2$ over $K_1$ and $K_2$ respectively. 

\begin{definition}
	We say that $X_1$ and $X_2$ are $H_{\ell}^*$-equivalent if there exists an isomorphism of graded algebras $H^*_{X_{1},\ell}\cong H^*_{X_{2},\ell}$ which is an isometry in degree $2$.
\end{definition} 
Note that if $X_{1}$ and $X_{2}$ are deformation equivalent then they are $H_{\ell}^*$-equivalent, since in this case the manifolds $X_{1,\CC}$ and $X_{2,\CC}$ are homeomorphic, and both the graded algebra $H^*_{X_i}$ and the Beauville-Bogomolov form on $H^2_{X_i}$ only depend on the topology of the complex manifold $X_{i,\CC}$.

\begin{proposition}\label{prop:extendisometry}
	Let $X_1$ and $X_2$ be $H_{\ell}^*$-equivalent hyper-K\"{a}hler varieties over $K_1$ and $K_2$ respectively. Assume given a subgroup $\Gamma\subset \Gal(\bar{K}_1/ K_1)$ and a homomorphism $\epsilon\colon \Gamma\to \Gal(\bar{K}_2/K_2)$; we let $\Gamma$ act on $H^*_{X_1,\ell}$ via $\sigma_{X_1}$ and on $H^*_{X_2,\ell}$ via $\epsilon\circ \sigma_{X_2}$. If there exists an isometry $f\colon H^2_{X_{1}, \ell}\to H^2_{X_2, \ell}$ which is $\Gamma$-equivariant, then there exists an isomorphism $F\colon H^*_{X_1. \ell} \to H^*_{X_2, \ell}$ of graded algebras whose degree $2$ component is again $\Gamma$-equivariant.
\end{proposition} 
\begin{proof}
	Since $X_1$ and $X_2$ are $H_{\ell}^*$-equivalent, there exists an isomorphism of graded algebras $\Psi\colon H^*_{X_{1},\ell}\to H^*_{X_{2},\ell}$ which is an isometry in degree $2$; let $\psi=\Psi^{(2)}$ be its component in degree $2$. 
	We have $\psi^{-1} \circ f \in \Oo(H^2_{X_1})(\QQ_{\ell})$. We may assume this isometry has determinant $1$, for if it has determinant $-1$ we can choose an ample line bundle on $X_1$ with first Chern class $e\in H^2_{X_1,\ell}$ and replace~$f$ with the isometry given by $e\mapsto -f(e)$ and $v\mapsto f(v)$ for any $v\in \langle e \rangle^{\bot}$, which is again $\Gamma$-equivariant.
	
	The morphism $\pi\colon \CSpin(H^2_{X_1})\to \SO(H^2_{X_1})$ is surjective on $\QQ_{\ell}$-points. Indeed, by Hilbert's theorem 90 (\cite[Chapter~X, \S1]{serrelocalfields}), the short exact sequence 
	\[
	1\to \mathbb{G}_m \to \CSpin(H^2_{X_1}) \to \SO(H^2_{X_1}) \to 1
	\]
	yields a short exact sequence
	\[
	1\to \QQ_{\ell}^{\times} \to \CSpin(H^2_{X_1})(\QQ_{\ell}) \to \SO(H^2_{X_1})(\QQ_{\ell}) \to 1.
	\] 
	Therefore, there exists $g\in \CSpin(H^2_{X_1})(\QQ_{\ell})$ such that $\pi(g)=\psi^{-1}\circ f$.  By Lemma \ref{lem:byauto}, $R(g)$ is a graded algebra automorphism of $H^*_{X_1, \ell}$. It follows that $F\coloneqq \Psi \circ R(g)\colon H^*_{X_1, \ell}\to H^*_{X_2,\ell}$ is an isomorphism of graded algebras. By Remark~\ref{rem:isogenies}, the degree $2$ component $F^{(2)}$ of $F$ is $\mathrm{Nm}(g) \cdot f$, and hence it is $\Gamma$-equivariant.  
\end{proof}

We can now prove Theorem \ref{thm:GalRepHK}, whose statement is recalled below.
\begin{theorem}\label{thm:main}
	Let $K_1, K_2$ be subfields of $\CC$, finitely generated over $\QQ$, and let~$X_i$ be a hyper-K\"{a}hler variety over $K_i$, for $i=1,2$. Assume that $X_1$ and $X_2$ are $H_{\ell}^*$-equivalent and that the Mumford--Tate conjecture holds for both of them. Let $\Gamma\subset \Gal(\bar{K}_1/K_1)$ be a subgroup, let $\epsilon\colon \Gamma\to \Gal(\bar{K}_2/K_2)$ be a homomorphism and let $f\colon H^2_{X_1, \ell} \to H^2_{X_2, \ell}$ be a $\Gamma$-equivariant isometry.
	Then, there exist a subgroup $\Gamma'\subset \Gamma$ of finite index and a $\Gamma'$-equivariant isomorphism of graded algebras $F\colon H^*_{X_1, \ell}\to H^*_{X_2, \ell}$.
\end{theorem}

\begin{proof}
	Replacing $K_i$ by a finite field extension if necessary, we may assume that $G(H^*_{X_i,\ell})$ is connected for $i=1,2$. Since the Mumford--Tate conjecture holds for $X_i$, by Theorem~\ref{thm:HodgeHK} the representation $\sigma_{X_i}\colon \Gal(\bar{K}_i/K_i)\to \GL(H^*_{X_i,\ell})$ factors through the $\QQ_{\ell}$-points of the image of the LLV-representation $R\colon \mathrm{G}_0({X_i}) \to \prod_j \GL(H^j_{X_i})$.
	
	By Proposition \ref{prop:extendisometry} there exists an isomorphism $F\colon H^*_{X_1, \ell}\to H^*_{X_2, \ell} $ of graded algebras whose degree $2$ component $F^{(2)}$ is $\Gamma$-equivariant. 
	Now the argument is the same as in the proof of Proposition \ref{prop:hodge}. We consider the isomorphism $F_*\colon \GL(H^*_{X_1,\ell})\to \GL(H^*_{X_2,\ell})$ given by $A\mapsto FAF^{-1}$, and the analogous isomorphism $F^{(2)}_*\colon \GL(H^2_{X_1,\ell})\to \GL(H^2_{X_2,\ell})$. 
	By Lemma \ref{lem:functoriality} the isomorphism $F_*$ restricts to an isomorphism $R\bigl({\mathrm{G}}_0(X_1)\bigr)(\QQ_{\ell})\cong R\bigl({\mathrm{G}}_0(X_2)\bigr)(\QQ_{\ell})$.
	
	We consider the diagram
	\[
	\begin{tikzcd}
	& \Gamma \arrow[hook']{dl} \arrow{dr}{\epsilon} \\
	\Gal(\bar{K}_1/K_1) \arrow{d}{\sigma_{X_1}} && \Gal(\bar{K}_2/K_2) \arrow{d}{\sigma_{X_2}} \\ 
	R\bigl({\mathrm{G}}_0(X_1)\bigr)(\QQ_{\ell}) \arrow{d}{\mathrm{pr}_1}  \arrow{rr}{F_*}  && R\bigl({\mathrm{G}}_0(X_2)\bigr)(\QQ_{\ell})  \arrow{d}{\mathrm{pr}_2} \\
	R^{(2)}\bigl({\mathrm{G}}_0(X_1)\bigr)(\QQ_{\ell}) \arrow{rr}{ F^{(2)}_* } && R^{(2)}\bigl({\mathrm{G}}_0(X_2)\bigr)(\QQ_{\ell})
	\end{tikzcd}
	\]
	We have to show that, up to replacing $\Gamma$ by a subgroup of finite index, this diagram commutes.
	
	Since $F^{(2)}$ is $\Gamma$-equivariant by assumption, we have 
	$
	F_*^{(2)} \circ \mathrm{pr}_1 \circ \sigma_{X_1} = \mathrm{pr}_2 \circ \sigma_{X_2} \circ \epsilon	$.
	By Remark \ref{rem:isogenies}, the homomorphism $\mathrm{pr}_2\colon R\bigl({\mathrm{G}}_0(X_2)\bigr) \to R^{(2)}\bigr({\mathrm{G}}_0(X_2)\bigr)$ is either an isomorphism or a central isogeny of degree $2$; let $C$ be its kernel. Then there exists a homomorphism $\chi\colon \Gamma \to C(\QQ_{\ell})$ such that $
	F_*\circ \sigma_{X_1}(\gamma) = \chi(\gamma) \cdot \sigma_{X_2}(\gamma)
	$
	for any $\gamma\in \Gamma$. The kernel $\Gamma'\subset \Gamma$ of $\chi$ is a subgroup of finite index such that $F$ is $\Gamma'$-equivariant.
\end{proof}
\begin{remark}
	Note that in the above proof we have only used that $G(H^*_{X_i,\ell})^0\subset \MT(H^*_{X_i})(\QQ_{\ell})$, so we only need one of the two inclusions predicted by the Mumford--Tate conjecture. 
\end{remark}

\subsection{}\label{subsec:setting}
We apply Theorem \ref{thm:main} to the study of Galois representations on the cohomology of hyper-K\"{a}hler varieties over finite fields. We will consider the following situation. Let $k$ be a finite field, and let $Z_1$ and $Z_2$ be smooth projective varieties over $k$. We assume that there exist hyper-K\"{a}hler varieties $X_1$ and~$X_2$ over fields of characteristic $0$ which lift $Z_1$ and $Z_2$. More precisely, we assume that there exist:
\begin{itemize}
	\item normal integral domains $R_i\subset \CC$ essentially of finite type over $\ZZ$ (i.e.~the localization of a finite type algebra over $\ZZ$) with fraction fields $K_i$ of characteristic~$0$;
	\item smooth and projective morphisms $\mathfrak{X}_i\to \Spec(R_i)$ whose generic fibres $X_i$ are hyper-K\"{a}hler;
	\item homomorphisms $R_i\to k$ together with isomorphisms $\mathfrak{X}_i\otimes _{R_i} k\cong Z_i$ of $k$-schemes.
\end{itemize}

We let $\ell$ be a prime number invertible in $k$ and consider $H^*_{Z_i, \ell}=\bigoplus_j H^{j}_{\et}(Z_{i, \bar{k}}, \QQ_{\ell})$. By the smooth and proper base-change theorems we have an isomorphism $H^*_{X_i, \ell}\cong H^*_{Z_i, \ell}$ of graded algebras. Via this isomorphism, the Beauville-Bogomolov form induces a non-degenerate symmetric bilinear form on $H^2_{Z_i, \ell}$.

\begin{remark}\label{rem:BBform}
	A priori, the bilinear form that we obtain on $H^2_{Z_i, \ell}$ depends on the choices of $R_i$ and $\mathfrak{X}_i$. However, by~\cite[\S4]{Huy06}, the formula
	\[
	\alpha \mapsto \int_{X_{i}} \alpha^2 \wedge \sqrt{\mathrm{td}(X_i)}
	\]
	defines a non-degenerate quadratic form on $H^2_{X_i}$ which is a non-zero multiple of the Beauville-Bogomolov form. The induced form on $H^2_{Z_i,\ell}$ is $\alpha \mapsto \int_{Z_{i}} \alpha^2 \wedge \sqrt{\mathrm{td}(Z_i)}$, and it is thus independent from the choices of $R_i$ and $\mathfrak{X}_i$.
\end{remark}

We now prove Theorem \ref{thm:GalRep2}.

\begin{theorem}\label{thm:Galfinitefields}
	With notations and assumptions as above, assume that $X_1$ and~$X_2$ are $H_{\ell}^*$-equivalent, and that the Mumford--Tate conjecture holds for both of them. Let $f\colon H^2_{Z_1, \ell} \to H^2_{Z_2, \ell}$ be a $\Gal(\bar{k}/k)$-equivariant isometry. Then, there exist a finite field extension $k'$ of $k$ and a $\Gal(\bar{k}/k')$-equivariant isomorphism of graded algebras $F\colon H^*_{Z_1, \ell} \to H^*_{Z_2, \ell}$.
\end{theorem}
\begin{proof}
	Let $\lvert k \rvert = p^r$, and let $\mathrm{Fr}_k\in \Gal(\bar{k}/k)$ be the Frobenius automorphism. With notations as in \S\ref{subsec:setting}, let $\mathfrak{m}_i\subset R_i$ be the kernel of $R_i\to k$; let $\lvert R_i/\mathfrak{m}_i\rvert= p^{r/a_i}$ and denote by $\phi_i\in \Gal(\bar{K}_i/K_i)$ a Frobenius element at $\mathfrak{m}_i$, for $i=1,2$. 
	
	By construction, we have isomorphisms $\langle \phi_i^{a_i} \rangle = \langle \mathrm{Fr}_k \rangle$ (both isomorphic to $\mathbb{Z}$) such that the action of $\phi^{a_i}$ on $H^*_{Z_i, \ell}$ via the base-change isomorphism $H^*_{X_i, \ell} \cong H^*_{Z_i, \ell}$ corresponds to that of $\mathrm{Fr}_k$.
	
	Let now $\Gamma= \langle \phi_1^{a_1}\rangle \subset \Gal(\bar{K}_1/K)$ and let $\epsilon\colon \Gamma\to \Gal(\bar{K}_2/K_2)$ be the homomorphism such that $\phi_1^{a_1}\mapsto \phi_2^{a_2}$. By Theorem \ref{thm:main}, there exists an integer $m$ and an isomorphism $H^*_{Z_1, \ell} \to H^*_{Z_2, \ell}$ of graded algebras which is $\mathrm{Fr}^{m}_k$-equivariant.
\end{proof}

\section{Motives}\label{sec:motivic}

\subsection{} Let $K\subset \CC$ be a subfield. We will work with the category of motives over $K$ introduced by Andr\'{e} in \cite{andre1996Motives}, which we denote by $\mathrm{AM}_K$. It is a neutral tannakian semisimple abelian category, with functors 
\[
\mathrm{SmProj}^{\mathrm{op}}_K \xrightarrow{\mathcal{H}^j} \mathrm{AM}_K \xrightarrow{r} \mathrm{HS}_{\QQ}^{\mathrm{pol}}
\]
such that $r\circ \mathcal{H}^{j}$ is the functor $H^{j}$ which associates to a smooth and projective variety~$X$ the rational Hodge structure~$H^{j}_X$. We denote by $\mathcal{H}^*_X\coloneqq \bigoplus_j \mathcal{H}^{j}_X$ the motive of the smooth and projective variety $X$.

The composition of the realization functor $r$ with the forgetful functor to $\QQ$-vector spaces is a fibre functor on $\mathrm{AM}_K$; the tensor automorphisms of this functor form a pro-reductive group $\mathcal{G}_{\mathrm{mot}}(\mathrm{AM}_K)$ over $\QQ$, and $\mathrm{AM}_K$ is equivalent to the category of finite dimensional $\QQ$-representations of $\mathcal{G}_{\mathrm{mot}}(\mathrm{AM}_K)$.


We denote by $\langle \mathcal{M} \rangle$ the tannakian subcategory of $\mathrm{AM}_K$ generated by a motive $\mathcal{M}$. Via restriction to $\langle \mathcal{M}\rangle$ of the fibre functor described above, we obtain a reductive algebraic group $\Gmot(\mathcal{M})\subset \GL\bigl(r(\mathcal{M})\bigr)$ whose category of representations is equivalent to $\langle \mathcal{M} \rangle$; we have a canonical surjective homomorphism 
$ \mathcal{G}_{\mathrm{mot}}(\mathrm{AM}_K)\to \Gmot(\mathcal{M})$. 
If $\mathcal{M}$, $\mathcal{N}\in\mathrm{AM}_K$, a linear isomorphism $f\colon r(\mathcal{M})\to  r(\mathcal{N})$ between their realizations comes from an isomorphism of motives if and only if it is $\mathcal{G}_{\mathrm{mot}}(\mathrm{AM}_K)$-equivariant.

\begin{remark} \label{rem:gmotFields}
	If $K'$ is any field extension of $K$, the group $\Gmot(\mathcal{M}_{K'})$ is a subgroup of finite index of $\Gmot(\mathcal{M})$, and there exists a finite field extension $K^{\dagger}$ of $K$ such that for any extension $K'/K^{\dagger}$ we have $\Gmot(\mathcal{M}_{K'})=\Gmot(\mathcal{M}_{K^{\dagger}})$, see \cite[\S3.1]{moonen17}. In particular $\Gmot(\mathcal{M}_{\CC})$ is a subgroup of finite index of $\Gmot(\mathcal{M})$.
\end{remark}



\subsection{} 
We will need to work with motives in families. For our purposes, it will be sufficient to treat the case in which the base field is $\mathbb{C}$. Let $S$ be a non-singular and connected complex quasi-projective variety. We will use the category $\mathrm{AM}/S$ of families of motives over $S$ defined by Moonen \cite{moonen17}. The category $\mathrm{AM}/S$ is a semisimple, neutral tannakian, abelian category, and there is a realization functor $r\colon \mathrm{AM}/S\to \mathrm{VHS}_{\QQ}/S$ to polarized variations of $\QQ$-Hodge structures over $S$.


The prototype of an object of $\mathrm{AM}/S$ is the motive $\mathcal{H}^*_{X/S}=\bigoplus_j \mathcal{H}^{j}_{X/S}$ of a smooth and projective morphism $\mathfrak{X}\to S$.
For any $s\in S$ the fibre of $\mathcal{H}^*_{\mathfrak{X}/S}$ at $s$ is the motive~$\mathcal{H}^*_{\mathfrak{X}_s}$ of the fibre. The realization of the motive $\mathcal{H}^*_{X/S}$ is the variation of Hodge structures $H^*_{\mathfrak{X}/S}$ over~$S$. In general, an object $\mathcal{M}/S\in \mathrm{AM}/S$ is cut out inside some $\mathcal{H}^*_{\mathfrak{X}/S}$ as above by a global section $p$ of the local system $\End\bigl(r(M)/S\bigr)$ such that $p\circ p=p$ and $p_s$ is a motivated cycle for all $s\in S$.



Let $\mathcal{M}/S\in \mathrm{AM}/S$ and let $M/S\in \mathrm{VHS}_{\QQ}/S$ be its realization. We denote by $\mathcal{M}_s$ and $M_s$ the fibre at $s$ of $\mathcal{M}/S$ and $M/S$ respectively.
Then there are local systems of algebraic groups over $S$
\[
G_{\mathrm{mono}}(M/S)^0 \trianglelefteq \MT(M/S) \subset \Gmot(\mathcal{M}/S) \subset \GL(M/S),
\]
where: 
\begin{itemize}
	\item $G_{\mathrm{mono}}(M/S)^0$ is the connected monodromy group: for each $s \in S $ the fibre $G_{\mathrm{mono}}(M/S)^0_s$ is the identity component of the Zariski closure of the image of the monodromy representation $ \pi_1(S,s) \to \GL(M_s)$;
	\item $\MT(M/S)$ is the generic Mumford--Tate group: for each $s\in S$ we have an inclusion $\MT(M_s)\subset \MT(M/S)_s$, and equality holds for very general~$s$, that is, for $s$ in the complement of a countable union of closed subvarieties of $S$;
	\item $\Gmot(\mathcal{M}_{s})$ is the generic motivic Galois group: for each $s\in S$ we have an inclusion $\Gmot(\mathcal{M}_{s})\subset \Gmot(\mathcal{M}/S)_s$, and equality holds for very general~$s$.
\end{itemize}

The key result needed to develop the above theory is Andr\'{e}'s deformation principle for motivated cycles \cite[Th\'eor\`eme 0.5]{andre1996Motives}.

\subsection{}
We now review some of the results on the Andr\'{e} motives of hyper-K\"{a}hler varieties which we obtained in \cite{FFZ}. To any complex hyper-K\"{a}hler variety $X$ with $b_2(X)>3$ we attached an algebraic group 
\[
P(X) \subset \Gmot(\mathcal{H}_X),
\]	
called the defect group of $X$, that is defined as follows. 

If the odd cohomology of $X$ is trivial, then $P(X)$ is simply defined as the kernel of the projection $\Gmot(\mathcal{H}^*_X) \to \Gmot(\mathcal{H}^2_X)$ corresponding to $\mathcal{H}^2_X\subset \mathcal{H}^*_{X}$. 

If instead $X$ has non-trivial cohomology in some odd degree, we need to consider the abelian variety $A$ obtained from~$H^2_X$ via the Kuga-Satake construction (\cite{deligne1971conjecture}) in order to control the odd part of the cohomology. In the case when $X$ has non-trivial odd cohomology we will assume the following:

\begin{assumption}\label{ass:ODD}
	The motive $\mathcal{H}^1_A$ belongs to the tannakian category of motives generated by $\mathcal{H}^*_X$.
\end{assumption}
 
 We then define the defect group  $P(X)$ as the kernel of the induced projection $\Gmot(\mathcal{H}^*_X)\to \Gmot(\mathcal{H}^1_{A})$. 
 We expect Assumption \ref{ass:ODD} to be always satisfied; for the only known class of hyper-K\"{a}hler varieties with odd cohomology, that is, the deformation class of generalized Kummer varieties, we verified this in \cite{FFZ}. In general, contrarily to what claimed in \textit{loc. cit.}, the question remains open, see the Erratum \cite{ErrataFFZ} where we show that this assumption is not necessary to define the defect group; Assumption~\ref{ass:ODD} may be translated into a question on the representation of the LLV-Lie algebra on the odd cohomology.
 
By construction, $P(X)$ acts on $H^*_X$ via graded algebra automorphisms and it acts trivially on $H^2_X$. We summarize below the main properties of the defect group $P(X)$, which are proven in \cite[Lemma 6.8, Theorems 6.9 and 6.12]{FFZ}.

\begin{theorem}\label{thm:propertiesdefectgroup}
	Let $X$ be a complex hyper-K\"{a}hler variety with $b_2(X)>3$. If $X$ has non-trivial cohomology in odd degree, assume that it satisfies Assumption \ref{ass:ODD}. 
	Then:
	\begin{enumerate}[label=(\alph*)]
		\item the action of $P(X)$ on $H^*_X$ commutes with the LLV-Lie algebra $\mathfrak{g}(X)$;
		\item we have $\Gmot(\mathcal{H}*_X)= P(X)\times \MT(H^*_X)$; 
		\item if $\mathfrak{X}\to S$ is a smooth and projective morphism to a non-singular connected variety~$S$ with fibres hyper-K\"{a}hler varieties with $b_2>3$, then $s\mapsto P(\mathfrak{X}_s)$ defines a local system of algebraic groups $P(\mathfrak{X}/S)\subset \GL(H^*_{\mathfrak{X}/S})$ over $S$.
	\end{enumerate}
\end{theorem}
If $\mathfrak{X}\to S$ is a family as in (c), the decomposition in (b) spreads over $S$ and gives a decomposition at the level of local systems of algebraic groups 
\[
\Gmot(\mathcal{H}^*_{\mathfrak{X}/S})=P(\mathfrak{X}/S) \times \MT(H^*_{\mathfrak{X}/S}).
\]
For all $s\in S$, the inclusion of $\Gmot(\mathcal{H}^*_{\mathfrak{X}_s})$ into $\Gmot(\mathfrak{X}/S)_s$ is the direct product of $\MT(H^*_{\mathfrak{X}_s})\subset \MT(H^*_{\mathfrak{X}/S})_s$ and $P(\mathfrak{X}_s)=P(\mathfrak{X}/S)_s$. Since the connected component of the identity of the monodromy group is a subgroup of the generic Mumford--Tate group of the family and the latter commutes with the defect group, the local system $P(\mathfrak{X}/S)$ becomes constant after a finite base change $S'\to S$.
\begin{remark}
	Conjecturally, the group $P(X)$ is trivial for any hyper-K\"{a}hler variety $X$ with $b_2>3$. In fact, the triviality of the defect group is equivalent to the conjecture which says that $\MT(H^*_X)=\Gmot(\mathcal{H}^*_X)$ (i.e. \textit{Hodge classes are motivated}), which would be a consequence of the Hodge conjecture.
\end{remark}

\subsection{} 
With these preliminaries behind us we can proceed towards the proof of Theorem \ref{thm:mot}. To start with, we will need a stronger version of the deformation invariance of defect groups from Theorem \ref{thm:propertiesdefectgroup}.(c). 

With $X$ as in Theorem \ref{thm:propertiesdefectgroup}, we have a canonical surjective homomorphism $\pi_X\colon \mathcal{G}_{\mathrm{mot}}(\mathrm{AM}_{\CC})\to \Gmot(\mathcal{H}^*_X)$.
We also let 
\[
\pr_X\colon \mathcal{G}_{\mathrm{mot}}(\mathrm{AM}_{\CC})\to P(X)
\]
be the composition of $\pi_X$ with the projection coming from the isomorphism $\Gmot(\mathcal{H}^*_X)=P(X)\times \MT(H^*_X)$. Via Tannaka duality, the projection $\pr_X$ corresponds to the subcategory of $\langle \mathcal{H}^*_{X}\rangle$ of motives on which $\MT(H^*_X)$ acts trivially, i.e.\ the motives in $\langle\mathcal{H}^*_{X}\rangle$ with realization a direct sum of trivial Hodge structures $\QQ(0)$.

\begin{remark}\label{rem:abelianpart}
	Composing $\pi_X$ with the other projection we obtain
	\[
	p_X\colon \mathcal{G}_{\mathrm{mot}}(\mathrm{AM}_{\CC})\to \MT(H^*_X).
	\]
	By the definition of the defect group this homomorphism corresponds via Tannaka duality to the subcategory $\langle \mathcal{H}^2_X \rangle$ of $\mathrm{AM}_{\CC}$ 
	if the odd cohomology of $X$ is trivial, and to the subcategory $\langle \mathcal{H}^1_A \rangle$ 
	otherwise, where $A$ is the Kuga-Satake abelian variety associated to~$H^2_X$. 
    Since $b_2(X)>3$, the work of Andr\'{e} (\cite{Andre1996}) ensures that the motive $\mathcal{H}^2_X$ is abelian, i.e.\ it belongs to the tannakian subcategory $\mathrm{AM}_{\CC}^{\mathrm{ab}}$ of $\mathrm{AM}_{\CC}$ generated by the motives of abelian varieties. 
	Hence, in any case the homomorphism $p_X$ factors through the quotient homomorphism  $\mathcal{G}_{\mathrm{mot}}(\mathrm{AM}_{\CC}) \to \mathcal{G}_{\mathrm{mot}}(\mathrm{AM}_{\CC}^{\mathrm{ab}})$, where $\mathcal{G}_{\mathrm{mot}}(\mathrm{AM}_{\CC}^{\mathrm{ab}})$ is the motivic Galois group of the tannakian category $\mathrm{AM}^{\mathrm{ab}}_{\CC}$ of abelian motives.
\end{remark}

\begin{proposition}\label{prop:constantdefect}
	Let $\mathfrak{X}\to S$ be a smooth projective family of hyper-K\"{a}hler varieties with $b_2>3$; if the odd cohomology of the fibres is not trivial, assume that they satisfy Assumption \ref{ass:ODD}.
	Assume that the monodromy group $G_{\mathrm{mono}}(H^*_{\mathfrak{X}/S})$ is connected. Let $a, b$ be points in~$S$. Choose a continous path $\gamma$ from $a$ to $b$ and let $\Xi\colon P(\mathfrak{X}_{a})\to P(\mathfrak{X}_{b})$ be the isomorphism obtained via parallel transport along $\gamma$ in the local system $P(\mathfrak{X}/S)$. Then $\Xi$ does not depend on the choice of $\gamma$ and 
	the diagram
	\[
	\begin{tikzcd} 
	& P(\mathfrak{X}_{a}) \arrow{dd}{\Xi}\\
	\mathcal{G}_{\mathrm{mot}}(\mathrm{AM}_{K})  \arrow{ur}{{\pr_{\mathfrak{X}_{a}}}} \arrow[swap]{dr}{{\pr_{\mathfrak{X}_{b}}}} \\
	& P(\mathfrak{X}_{b})
	\end{tikzcd} 
	\]
	is commutative. 
\end{proposition}
\begin{proof}
	Since the monodromy group is connected by assumption, the local system $P(\mathfrak{X}/S)$ is constant and $\Xi$ does not depend on the choice of the path~$\gamma$.
	
	Consider any motive over $S$ of the form 
	\[
	\mathcal{T}/S\coloneqq (\mathcal{H}^*_{\mathfrak{X}/S})^{\otimes t_1 } \otimes (\mathcal{H}^*_{\mathfrak{X}/S})^{\vee, \otimes t_2} \otimes \QQ_{S}(t_3) \in \mathrm{AM}/S,
	\]
	for integers $t_1,t_2,t_3$. Let $T/S$ be its realization. For any $s\in S$ we let $W_{s} \subset T_{s}$ be the subspace of invariants for the generic Mumford--Tate group $\MT(T/S)_s$; this yields a sub-variation of Hodge structures $W/S\subset T/S$. Moreover, as $\MT(T/S)_s$ is normal in $\Gmot(\mathcal{T}/S)_s$, the variation $W/S$ is the realization of a submotive $\mathcal{W}/S\subset \mathcal{T}/S$ over~$S$.
	
	The motive $\mathcal{W}/S$ is a constant motive over $S$. Indeed, let us denote by~$\mathcal{D}$ the motive $\mathcal{W}_{b}$, and let ${\mathcal{D}}/S$ be the constant motive over $S$ with fibre $\mathcal{D}$; let $D/S$ be the realization of $\mathcal{D}/S$.
	Then $\mathrm{id}_{b}\colon {W}_{b}\to {D}_{b}$ is monodromy invariant and obviously an isomorphism of motives; by \cite[Th\'eor\`eme 0.5]{andre1996Motives} it extends to an isomorphism of families of motives $\mathcal{W}/S \xrightarrow{\ \sim \ } \mathcal{D}/S$.
	
	It follows that parallel transport along $\gamma$ in the local system $W/S$ gives a linear map $\Psi\colon W_a\to W_b$ which is 
	the realization of an isomorphism of motives $\mathcal{W}_{a}\cong \mathcal{W}_{b}$. 
	Hence the induced isomorphism $\Psi_*\colon \GL(W_{a})\to \GL(W_{b})$ fits into a commutative diagram
	\[
	\begin{tikzcd} 
	& \Gmot(\mathcal{W}_{a}) \arrow{dd}{\Psi_*} \\
	\mathcal{G}_{\mathrm{mot}}(\mathrm{AM}_{\CC})  \arrow{ur} \arrow[swap]{dr} \\
	& \Gmot(\mathcal{W}_{b})
	\end{tikzcd} 
	\]	 
	Note that since the generic Mumford--Tate group acts trivially on $W_s$ by construction, the group $\Gmot(\mathcal{W}_s)$ is a quotient of the defect group $P(\mathfrak{X}_s)$.
	
	We now choose the tensor construction $T/S$ in such a way that the action of $P(\mathfrak{X}_s)$ on the subspace $W_s$ is faithful; in this case we have $\Gmot(\mathcal{W}_s)=P(\mathfrak{X}_s)$ for all points $s\in S$, and the homomorphism $\mathcal{G}_{\mathrm{mot}}(\mathrm{AM}_{\CC})\to \Gmot(\mathcal{W}_{s})$ is identified with $\pr_{\mathfrak{X}_s}\colon \mathcal{G}_{\mathrm{mot}}(\mathrm{AM}_{\CC})\to P(\mathfrak{X}_s)$.
	Moreover, $P(\mathfrak{X}/S)\subset \GL(W/S)$ is a sub-local system of algebraic groups, and therefore the isomorphism $\Xi$ obtained via parallel transport along $\gamma$ in the local system $P(\mathfrak{X}/S)$ is the restriction of the isomorphism $\Psi_*\colon \GL(\mathcal{W}_{a})\to \GL(\mathcal{W}_{b})$ to $P(\mathfrak{X}_{a})$. This concludes the proof.
\end{proof}

By Remark \ref{rem:gmotFields}, Theorem \ref{thm:mot} is equivalent to the following statement.
\begin{theorem}\label{thm:mot1}
	Let $X_1$ and $X_2$ be deformation equivalent complex hyper-K\"{a}hler varieties with $b_2>6$; if they have non-trivial odd cohomology, assume that they satisfy Assumption \ref{ass:ODD}.
	Assume that $f\colon H^2_{X_1}\to H^2_{X_2}$ is a Hodge isometry. Then there exists an isomorphism of graded algebras $F\colon H^*_{X_1}\to H^*_{X_2}$ which is the realization of an isomorphism of motives $\mathcal{H}^*_{X_1}\to \mathcal{H}^*_{X_2}$. 
\end{theorem}
\begin{proof} 
	By Theorem \ref{thm:connectingHK} there exist smooth and projective families of hyper-K\"{a}hler varieties $\mathfrak{X}^i\to S_i$ over non-singular connected quasi-projective varieties $S_i$, for $i=1,\dots , N$, and points $a_i, b_i \in  S_i$ with isomorphisms 
	\[
	X_1 \xrightarrow{\ \sim \ }\mathfrak{X}^1_{a_1}, 
	\ \ \ \ \
	\mathfrak{X}^i_{b_i} \xrightarrow{\ \sim \ } \mathfrak{X}^{i+1}_{a_{i+1}}, \ \text{for} \ i=1,\dots ,N-1,
	\ \ \ \ \
	\mathfrak{X}^N_{b_N} \xrightarrow{\ \sim \ } X_2.
	\]
	We may and will assume that the monodromy groups $G_{\mathrm{mono}}(H^*_{\mathfrak{X}^{i}/{S_i}})$ are connected. 
	For each $i$, let $\gamma_i $ be a path in $S_i$ from $a_i $ to $b_i$; let $\Psi \colon H^*_{X_1} \to H^*_{X_2}$ be the isomorphism obtained as composition of the isomorphisms $\Psi_i$ given by parallel transport along $\gamma_i$. We denote by $\psi\coloneqq \Psi^{(2)}\colon H^2_{X_1}\to H^2_{X_2}$ the isometry induced by~$\Psi$.
	
	Let $f\colon H^2_{X_1}\to H^2_{X_2}$ be a Hodge isometry. We construct the isomorphism of graded algebras $F\colon H^*_{X_1}\to H^*_{X_2}$ as in the proof of Proposition \ref{prop:extendisometry}: we may assume that the isometry $\psi^{-1}\circ f\colon H^2_{X_1}\to H^2_{X_1}$ has determinant $1$; thanks to Hilbert's Theorem 90 and the short exact sequence
	\[
	1\to \mathbb{G}_m \to \CSpin(H^2_{X_1})\xrightarrow{\ \pi  \ } \SO(H^2_{X_1}) \to  1,
	\] 
	the morphism $\pi\colon \CSpin(H^2_{X_1})\to \SO(H^2_{X_1})$ is surjective on $\QQ$-points and hence there exists $g\in \CSpin(H^2_{X_1})(\QQ)$ such that $\pi(g)=\psi^{-1}\circ f$. By Lemma \ref{lem:byauto}, $R(g)$ is a graded algebra automorphism of $H^*_{X_1}$, and we define 
	\begin{equation*}
	F\coloneqq \Psi\circ R(g)\colon H^*_{X_1}\to H^*_{X_2}.
	\end{equation*}
	
	We claim that $F$ is the realization of an isomorphism of motives $\mathcal{H}^*_{X_1}\to \mathcal{H}^*_{X_2}$.
	If $F_*\colon \GL(H^*_{X_1})\to \GL(H^*_{X_2})$ denotes the induced isomorphism, we have to prove that $F_*$ restricts to an isomorphism $\Gmot(\mathcal{H}^*_{X_1})\xrightarrow{\ \sim \ }  \Gmot(\mathcal{H}^*_{X_2})$ such that the diagram
	\[
	\begin{tikzcd}
	& \Gmot(\mathcal{H}^*_{X_1}) \arrow{dd}{F_*} \\
	\mathcal{G}_{\mathrm{mot}}(\mathrm{AM}_{\CC})\arrow{ur}{\pi_{1}} \arrow[swap]{dr}{\pi_{2}} \\
	& \Gmot(\mathcal{H}^*_{X_2})
	\end{tikzcd} 
	\]
	is commutative.

	By Theorem \ref{thm:propertiesdefectgroup}.(a), the automorphism $R(g)$ of $H^*_{X_1}$ commutes with $P(X_1)$. Since~$\Psi$ is the composition of parallel transport operators and the defect group is constant in families, $F_*$ identifies $P(X_1)$ with $P(X_2)$, and the induced isomorphism $F_*\colon P(X_1)\to P(X_2)$ is the restriction to $P(X_1)$ of~$\Psi_*\colon \GL(H^*_{X_1})\to \GL(H^*_{X_2})$. 
	
	
	
	Since the degree~$2$ component of $F$ is a multiple of the Hodge isometry $f$, by Proposition \ref{prop:hodge}, $F$ is an isomorphism of Hodge structures, and hence $F_*$ identifies $\MT(H^*_{X_1})$ with $\MT(H^*_{X_2})$.
	By Theorem \ref{thm:propertiesdefectgroup}.(b), $\Gmot(\mathcal{H}^*_{X_i})=P(X_i)\times \MT(H^*_{X_i})$,  and
 it suffices to show the commutativity of the two diagrams
	\[
	\begin{tikzcd} 
	& P(X_1) \arrow{dd}{ \Psi_* }  & & & \MT({H}^*_{X_1}) \arrow{dd}{ F_* } \\
	\mathcal{G}_{\mathrm{mot}}(\mathrm{AM}_{\CC})  \arrow{ur}{\pr_1} \arrow[swap]{dr}{\pr_2} & & & \mathcal{G}_{\mathrm{mot}}(\mathrm{AM}_{\CC})  \arrow{ur}{ p_1} \arrow[swap]{dr}{p_2} \\
	& P(X_2)  & & & \MT({H}^*_{X_2}) 
	\end{tikzcd} 
	\]
	
	The left triangle is commutative by repeated application of Proposition \ref{prop:constantdefect}, since the restriction of $\Psi_*$ to $P(X_1)$ is the composition of the isomorphisms $\Xi_i$ obtained via parallel transport along $\gamma_i$ in the local system $P(\mathfrak{X}^i/S_i)$.
	
	For the right one, we proceed as follows. By Remark \ref{rem:abelianpart}, $p_i$ corresponds to the inclusion of a subcategory of abelian motives; equivalently, for $i=1,2$, the homomorphism $p_i$ factors through $\mathcal{G}_{\mathrm{mot}}(\mathrm{AM}_{\CC})\to \mathcal{G}_{\mathrm{mot}}(\mathrm{AM}_{\CC}^{\mathrm{ab}})$. If we let $\mathrm{HS}_{\QQ}^{\mathrm{ab}}\subset \mathrm{HS}_{\QQ}^{\mathrm{pol}}$ be the tannakian subcategory generated by the Hodge structures of abelian varieties, we have $\mathcal{MT}(\mathrm{HS}_{\QQ}^{\mathrm{ab}})=\mathcal{G}_{\mathrm{mot}}(\mathrm{AM}_{\CC}^{\mathrm{ab}})$, by \cite[Th\'eor\`eme 0.6.4]{andre1996Motives}. 
	But now the diagram
	\[
	\begin{tikzcd}
	&  \MT(H^*_{X_1}) \arrow{dd}{F_*} \\
	\mathcal{G}_{\mathrm{mot}}(\mathrm{AM}_{\CC}^{\mathrm{ab}}) = \mathcal{MT}(\mathrm{HS}_{\QQ}^{\mathrm{ab}}) \arrow{ur}{p'_1} \arrow[swap]{dr}{p'_2}\\
	& \MT(H^*_{X_2}) 
	\end{tikzcd}
	\]
	is commutative, since $F$ is an isomorphism of Hodge structures. 
\end{proof}

Finally, we prove Corollary \ref{cor:unconditionalGalRep} from the introduction.

\begin{corollary}\label{cor:unconditionalGalRep1}
	Let $K\subset \CC$ be a subfield which is finitely generated over $\QQ$, and let $X_1$, $X_2$ be deformation equivalent hyper-K\"{a}hler varieties over $K$ with $b_2(X_i)>6$. If their odd cohomology is not trivial, assume that Assumption \ref{ass:ODD} holds for $X_1$ and $X_2$. Assume that $f\colon H^2_{X_1, \ell}\to H^2_{X_2,\ell}$ is a $\Gal(\bar{K}/K)$-equivariant isometry. Then, there exist a finite field extension $K'/K $ and an isomorphism $F\colon H^*_{X_1,\ell}\to H^*_{X_2,\ell}$ of graded algebras which is $\Gal(\bar{K}/K')$-equivariant.
\end{corollary}
\begin{proof}
	By \cite{Andre1996}, the Mumford--Tate conjecture in degree $2$ holds for $X_1$ and~$X_2$, and the motives $\mathcal{H}^2_{X_1}$ and $\mathcal{H}^2_{X_2}$ are abelian. Hence, there exists a finite extension $K'$ of $K$ such that the isometry $f$ is the realization of an isomorphism of motives over $K'$ with $\QQ_{\ell}$-coefficients. The same argument as in the proof of Theorem~\ref{thm:mot1} produces an isomorphism of graded algebras $F\colon H^*_{X_1,\ell}\to H^*_{X_2,\ell}$ which, up to further replacing $K'$ with a finite extension, is the realization of an isomorphism of motives over $K'$ with $\QQ_{\ell}$-coefficients. Hence, $F$ is $\Gal(\bar{K}/K')$-equivariant. 
\end{proof}

\section{Projective families of hyper-K\"{a}hler varieties}\label{sec:ProjFamilies}

In the proof of Theorem \ref{thm:mot} we used the following result.

\begin{theorem}\label{thm:connectingHK}
	Let $X_1$, $X_2$ be deformation equivalent complex projective hyper-K\"{a}hler varieties. Assume that $b_2(X_i)> 6$. Then there exist:
	\begin{itemize}
		\item smooth and projective families $\mathfrak{X}^i\to S_i$ with fibres hyper-K\"{a}hler varieties over connected and non-singular quasi-projective complex varieties $S_i$, for $i=1, \dots, N$;
		\item points $a_i, b_i\in S_i$, for $i=1,\dots, N$,  together with isomorphisms 
		\[
		X_1 \xrightarrow{\ \sim \ } \mathfrak{X}^1_{a_1}, 
		\ \ \ \ \
		\mathfrak{X}^i_{b_i}\xrightarrow{\ \sim \ } \mathfrak{X}^{i+1}_{a_{i+1}}, \ \text{for} \ i=1,\dots ,N-1,
		\ \ \ \ \
		\mathfrak{X}^N_{b_N} \xrightarrow{\ \sim \ } X_2.
		\]
	\end{itemize}
\end{theorem}
If $X_1$ and $X_2$ satisfy the conclusion of the Theorem, we will write $X_1\sim X_2$.

\begin{remark}
	In similar spirit, Soldatenkov \cite{soldatenkov19} shows that, under the assumption that $b_2>3$, the varieties $X_1$ and $X_2$ can be joined via smooth and proper (but not necessarily projective) families over curves; however, the total space of such a family is not an algebraic variety in general.
\end{remark}

\subsection{} Before giving the proof of Theorem \ref{thm:connectingHK} we recall some facts on polarized hyper-K\"{a}hler varieties. 
The positive cone of a hyper-K\"{a}hler manifold~$X$ is the connected component of $\{ x\in H^{1,1}(X,\RR) \text{ such that } (x,x)>0\}$ containing the K\"{a}hler cone, where $(\cdot ,\cdot)$ is the Beauville--Bogomolov pairing. 
We denote by $\mathrm{NS}^+(X)\subset \mathrm{NS}(X)$ the intersection of the positive cone with the N\'eron-Severi group.
In analogy to the case of K3 surfaces and $-2$-classes, the ample cone $\mathrm{Amp}(X)\subset \mathrm{NS}^+(X)$ of a projective hyper-K\"{a}hler manifold can be described in terms of so-called MBM classes~\cite{AV1} (this notion is equivalent to that of wall divisors introduced in \cite{mongardino}). 
The proof of Theorem \ref{thm:connectingHK} relies on the following result due to Amerik and Verbitsky. Let $\mathrm{MBM}(X)\subset \NS(X)$ be the subset of MBM classes on $X$.

\begin{theorem}\phantomsection \label{thm:AV}
	\begin{enumerate}[label=(\roman*)]
		\item \cite{AV1} Let $X$ be a projective hyper-K\"{a}hler manifold. The ample cone of $X$ is one of the connected components of 
		\[
		\mathrm{NS}^+(X)\setminus \bigcup_{z\in\mathrm{MBM}(X)} z^{\bot}.
		\]
		In particular, if $\mathrm{MBM}(X)=\emptyset$ then $\mathrm{Amp}(X)=\mathrm{NS}^+(X)$.
		\item \cite{AV} Fix a deformation class of hyper-K\"{a}hler manifolds with $b_2\geq 5$.
		There exists a positive integer $N$, depending only on the deformation class, such that for any projective~$X$ of the given deformation type, every MBM class $z$ on $X$ satisfies $$-N<(z,z)<0.$$
	\end{enumerate} 
\end{theorem}

\subsection{}
Let $X$ be a hyper-K\"{a}hler manifold and let $\Lambda$ be a lattice isometric to~$H^2(X,\ZZ)$ equipped with the Beauville--Bogomolov form. Let 
$$ 
\Omega = \{ x\in \PP(\Lambda\otimes \CC)\ | \ (x,x)=0, \ (x,\bar{x}) > 0 \}
$$
be the period domain. 
We fix a connected component $\mathfrak{M}$ of the moduli space of~$\Lambda$-marked hyper-K\"{a}hler manifolds containing $X$ (with a chosen marking). By the global Torelli theorem \cite{huybrechts2011global}, \cite{verbitskyTorelli}, the period map $\mathcal{P}\colon \mathfrak{M}\to \Omega$ is surjective with finite fibres, and each fibre consists of bimeromorphic hyper-K\"{a}hler manifolds. 

By Huybrechts' projectivity criterion \cite{Huy99}, a hyper-K\"{a}hler manifold~$Y$ is projective if and only if $\mathrm{NS}(Y)$ contains a class $h$ with $(h,h)>0$.
For any positive class~$h\in\Lambda$, we have a hypersurface 
$$
\Omega_{h^{\bot}}=\{x\in\Omega \ | \ (x,h)=0 \} \subset \Omega.
$$
The period space~$\Omega_{h^{\bot}}$ has two connected components; we denote by $\Omega^+_{h^{\bot}}$ the component parametrizing those $(Y,\tau)\in\mathfrak{M}$ such that~$\tau^{-1}(h)$ belongs to the positive cone.

We let $\mathfrak{M}_{h^{\bot}}^{\mathrm{a}}\subset \mathfrak{M}$ be the subset consisting of $(Y,\tau)$ such that $\tau^{-1}(h)$ represents an ample class on $Y$. Then $\mathfrak{M}_{h^{\bot}}^{\mathrm{a}}$ is Hausdorff and connected, and its image $\Omega_{h^{\bot}}^{\mathrm{a}}$ via the period map is open and dense in $\Omega_{h^{\bot}}^{+}$, by  \cite[Corollary~7.3]{markman2011survey}.

\subsection{}
The rest of the article is devoted to the proof of Theorem \ref{thm:connectingHK}. Let $X$ be a hyper-K\"{a}hler manifold and assume that $b_2(X)>6$. We fix a connected component~$\mathfrak{M}$ of the moduli space of $\Lambda$-marked hyper-K\"{a}hler manifolds.

We will show that given any $(X_1,\tau_1)$ and $(X_2,\tau_2)$ in $\mathfrak{M}$ with $X_1$ and $X_2$ projective, then~$X_1\sim X_2$, in the notation of the Theorem.
We start with the following case.
\begin{proposition}\label{prop:universalPHK}
	Let $h\in\Lambda$ be a positive class, and let $(X_1,\tau_1)$, $(X_2,\tau_2)$ be points of $\mathfrak{M}_{h^{\bot}}^{\mathrm{a}}$. Then~$X_1\sim X_2$.
\end{proposition}
\begin{proof}
	 Let $(Y,L)$ be a hyper-K\"{a}hler variety equipped with an ample divisor $L$. 
	 Following Andr\'e \cite[\S3.3]{Andre1996} (see also \cite{gritsenko2010moduli}) there exists a polarized deformation $\mathfrak{Y}\to S$ of $(Y, L)$ which is a smooth and projective family of hyper-K\"{a}hler varieties over a non-singular and connected quasi-projective variety $S$, with a distinguished fibre $\mathfrak{Y}_s=Y$ and the following property: denoting by~$\tilde{S}\to S$ the universal covering of~$S$, the image of the period map $\tilde{S}\to\Omega^{\mathrm{a}}_{h^{\bot}}$ contains an open subset. Here, $h=\phi(c_1(L))$ for a fixed marking $\phi$ on $Y$.
	 Upon replacing $S$ with a finite cover, we find a torsion free arithmetic subgroup $\Gamma\subset \mathrm{O}(h^{\bot})$ acting freely on $\Omega^+_{h^{\bot}}$ such that the period map descends to $\Psi\colon S\to \Gamma\backslash \Omega_{h^{\bot}}^+$. By \cite{borel}, $ \Gamma\backslash \Omega^+_{h^{\bot}}$ is a non-singular quasi-projective variety, and the map $\Psi$ is a dominant morphism of algebraic varieties.
	
	Let now $(X_1,\tau_1)$ and $(X_2,\tau_2)$ be as in the statement of the proposition. 
	Consider the respective polarized deformations $\mathfrak{X}_1\to S_1$ and $\mathfrak{X}_2\to S_2$ of $(X_1, \tau_{1}^{-1}(h))$ and $(X_2, \tau_{2}^{-1}(h))$ described above. For a suitable torsion free arithmetic subgroup $\Gamma\subset \mathrm{O}(h^{\bot})$ and finite covers of $S_1$ and $S_2$, we have dominant period maps~$\Psi_1\colon S_1\to \Gamma\backslash\Omega^+_{h^{\bot}}$ and $\Psi_2\colon S_2\to \Gamma\backslash\Omega^+_{h^{\bot}}$.
	As $\Omega^+_{h^{\bot}}$ is connected, $\Psi_1(S_1)\cap \Psi_2(S_2)$ is not empty. By the surjectivity of the period map, there exists $(Y,\tau)\in \mathfrak{M}$ whose period gives a point in this intersection via the quotient map $\Omega^+_{h^{\bot}}\to \Gamma\backslash \Omega^+_{h^{\bot}}$; if $(Y,\tau)$ is very general with the above  property, then $Y$ is of Picard rank $1$. 
	In this case $(Y,\tau)$ is the unique point in the fibre of the period map containing it, and it belongs to~$\mathfrak{M}^{\mathrm{a}}_{h^{\bot}}$. It follows that there exist $s_1\in S_1$ and $s_2\in S_2$ and an isomorphism between the fibres $\mathfrak{X}_{1,s_1} \cong Y \cong \mathfrak{X}_{2,s_2}$.
\end{proof}


The next case is the key step in the proof. 

\begin{proposition} \label{prop:firstPart}
	Let $h_1$ and $h_2$ be positive classes in $\Lambda$ such that the lattice~$\langle h_1,h_2\rangle $ is of signature $(1,1)$ and $(h_1,h_2)>0$. Assume that $(X_1,\tau_1) \in \mathfrak{M}_{h_1^{\bot}}^{\mathrm{a}}$ and $(X_2,\tau_2) \in \mathfrak{M}_{h_2^{\bot}}^{\mathrm{a}}$. Then $X_1\sim X_2$. 
\end{proposition}
\begin{proof}	
Let $h_1$ and $h_2$ be as above. Then $\Omega^+_{h_1^{\bot}}\cap \Omega^+_{h_2^{\bot}}\neq \emptyset$. If $\mathfrak{M}_{h_1^\bot}^{\mathrm{a}}\cap \mathfrak{M}_{h_2^{\bot}}^{\mathrm{a}}$ is also not empty, we can directly apply Proposition \ref{prop:universalPHK} to conclude; however in general this is not the case, and we need to modify $h_1$ and $h_2$ before applying that proposition.

We fix the constant~$N$ given by Theorem \ref{thm:AV} for our deformation type. First of all, we replace $h_1$ and $h_2$ with classes of which we can control the square. There exist $(Y_1, \phi_1) \in \mathfrak{M}_{h_1^{\bot}}^{\mathrm{a}}$ and $(Y_2, \phi_2)\in \mathfrak{M}_{h_2^{\bot}}^{\mathrm{a}}$ of maximal Picard rank. Then $\NS(Y_i)$ is an indefinite lattice of rank $b_2-2\geq 5$, and hence it contains a non-zero isotropic vector~$\phi_i^{-1}(y_i) \in \NS(Y_i)$, for $i=1,2$. 

\begin{lemma} 
	\label{claim:firstMove}
	There exist a prime number $p>N$ congruent to $3$ modulo $4$, an odd integer $j\gg 0$ and positive classes $l_1$ and $l_2$ in $\Lambda$ such that:
	\begin{itemize}
		\item $\phi_1^{-1}(l_1) $ (resp.~$\phi_2^{-1}(l_2) $) represents an ample divisor on $Y_1$ (resp.~$Y_2$);
		\item  $(l_1, l_1)=p^j f_1$ and $(l_2, l_2)=p^j f_2$, with $f_1$ and $f_2$ not divisible by $p$ and such that $f_1$ and $f_2$ are both quadratic residues modulo $p$;
		\item the lattice generated by $l_1$ and $l_2$ has signature $(1,1)$, and $(l_1, l_2)>0$.
	\end{itemize}
\end{lemma}
\begin{proof}
	Choose a large prime $p>N$ congruent to $3$ modulo $4$ and which does not divide any of the integers $(h_1, h_1)$, $(h_2, h_2)$, $(h_1, y_1)$ and $(h_2, y_2)$. Replacing $y_i$ with a suitable multiple, we may assume that $2(h_1, y_1)$ and $2(h_2, y_2)$ are (non-zero) quadratic residues modulo~$p$.
	Consider
	\[
	l_1\coloneqq p^jh_1 + y_2 \ \ \ \text{and} \ \ \ l_2\coloneqq p^jh_2 + y_2.
	\]
	Then we have
	\begin{align*}
	(l_1, l_1)& =p^j \Bigl(p^j(h_1, h_1) + 2 (h_1, y_1)\Bigr)\eqqcolon p^jf_1, \\
	  (l_2, l_2)& =p^j\Bigl(p^j(h_2,h_2) +  2 (h_2, y_2)\Bigr)\eqqcolon p^jf_2.
	\end{align*}
	For a large enough odd integer $j$, the class $\phi_1^{-1}(l_1) $ (resp.~$\phi_2^{-1}(l_2) $) represents an ample divisor on $Y_1$ (resp.~$Y_2$). Moreover $\langle l_1, l_2\rangle $ has signature $(1,1)$ and $(l_1, l_2)>0$. Therefore $l_1$ and $l_2$ have the required properties.
\end{proof}

Thanks to Proposition \ref{prop:universalPHK} it suffices to prove Proposition \ref{prop:firstPart} for $l_i, Y_i, \phi_i$ as above in place of $h_i, X_i, \tau_i$.

\begin{lemma}\label{claim:keyLemma}
	There exists vectors $v_1\in l_1^{\bot}$ and $v_2\in l_2^{\bot}$ such that:
	\begin{itemize}
		\item $(v_1,v_1)= p^{r_1}\epsilon_1$, where $r_1$ is an odd natural number and $\epsilon_1$ is a negative integer not divisible by $p$ and it is a quadratic residue modulo $p$;
		\item $(v_2,v_2)= p^{r_2}\epsilon_2$, where $r_2$ is an odd natural number and $\epsilon_2$ is a negative integer not divisible by $p$ and it is a quadratic residue modulo $p$;
		\item $(v_1,v_2)=0 = (l_2,v_1)$.
	\end{itemize}	
\end{lemma} 
\begin{proof}
	The orthogonal to $\langle l_1, l_2 \rangle$ in $\Lambda$ is an indefinite lattice of rank at least $5$; hence, we can find isotropic vectors $u_1$, $u_2\in\langle l_1, l_2 \rangle^{\bot}$ such that $(u_1,u_2)= t < 0$. The self-intersection of a linear combination $au_1+bu_2$ is $2abt$. It is then clear that a suitable linear combination $v_1$ of $u_1$ and $u_2$ will satisfy $(v_1,v_1)=p^{r_1}\epsilon_1$ for some odd natural number $r_1$ and a negative integer $\epsilon_1$ which is a non-zero square modulo $p$.
	
	The orthogonal to $\langle l_2, v_1\rangle$ is an indefinite lattice of rank at least $5$; hence, it contains isotropic vectors $z_1$, $z_2\in\langle l_2, v_1 \rangle^{\bot}$ such that $(z_1, z_2) = s < 0 $. As above, a suitable linear combination $v_2$ of $z_1$ and $z_2$ satisfies $(v_2,v_2)=p^{r_2}\epsilon_2$, for some odd natural number~$r_2$ and a negative integer $\epsilon_2$ which is a non-zero square modulo $p$.
\end{proof}

Consider now the rank $2$ sub-lattices of $\Lambda$:
$$L_1=\langle l_1, v_1 \rangle \ \ \ \text{and} \ \ \ L_2=\langle l_2, v_2\rangle.$$ 
The proof of the next Claim \ref{claim:lemma1} and Claim \ref{claim:lemma2} below will be given in \S\ref{subsec:technical}.

\begin{claim} \label{claim:lemma1}
	For $i=1$ or $2$, let $v\in L_i\otimes \QQ$. If $(v,v)\in \ZZ$, then $p $ divides $(v,v)$.
\end{claim} 

Since the lattices $L_1$ and $L_2$ are of signature~$(1,1)$, 
there exist $(W_1,\psi_1)$ and $(W_2,\psi_2)$ in~$\mathfrak{M}$ such that~$\mathrm{NS}(W_1)=\psi^{-1}_1(L_1)$ and~$\mathrm{NS}(W_2)=\psi^{-1}_2(L_2)$ and $\psi_1^{-1}(l_1)$ (resp.\ $\psi_2^{-1}(l_2)$) belongs to $\NS^+(W_1)$ (resp.~$\NS^+(W_2)$). Claim \ref{claim:lemma1} ensures that~$\mathrm{NS}(W_1)$ and~$\NS(W_2)$ contain no MBM classes; hence, by Theorem \ref{thm:AV},
$$\mathrm{Amp}(W_1)=\mathrm{NS}^+(W_1) \ \ \ \text{and} \ \ \ \mathrm{Amp}(W_2)=\mathrm{NS}^+(W_2).$$
In particular, $\psi_1^{-1}(l_1)$ (resp.~$\psi_2^{-1}(l_2)$) represents an ample class on $W_1$ (resp.~$W_2$).

For $k>0$, we define $w_{1,k} \in L_1$ and $w_{2,k}\in L_2$ as:
\[
w_{1,k}\coloneqq p^k l_1+v_1 \ \ \ \ \ \ w_{2,k}\coloneqq p^k l_2+v_2.
\]
For $k$ large enough, $\psi_1^{-1}(w_{1,k})$ (resp.~$\psi_2^{-1}(w_{2,k})$) represents an ample class on $W_1$ (resp.~on $W_2$). 
We let $S_k\subset \Lambda$ be the lattice generated by $w_{1,k}$ and $w_{2,k}$. 

\begin{claim}\label{claim:lemma2}
	Let $v\in S_k\otimes \QQ$, for $k\gg 0$. If $(v,v)\in \ZZ$, then $p$ divides $(v,v)$.
\end{claim}

If $k\gg 0$, the lattice $S_k$ has signature $(1,1)$, because this is the signature of the lattice generated by $l_1$ and $l_2$.
By the surjectivity of the period map, there exists $(Z,\eta)\in\mathfrak{M}$ such that $\mathrm{NS}(Z)=\eta^{-1}(S_k)$ and both the classes~$\eta^{-1}(w_{1,k})$ and $\eta^{-1}(w_{2,k})$ belong to $\NS^+(Z)$; this is possible because $(w_{1,k},w_{2,k})>0$ for $k\gg 0$.

By Claim~\ref{claim:lemma2}, there are no MBM classes in~$\mathrm{NS}(Z)$; hence, $\mathrm{Amp}(Z)=\mathrm{NS}^+(Z)$, by Theorem \ref{thm:AV}.
We therefore obtain:
\begin{equation*}
     (W_1,\psi_1) \in\mathfrak{M}_{l_1^{\bot}}^{\mathrm{a}}\cap \mathfrak{M}_{w_{1,k}^{\bot}}^{\mathrm{a}}, \ \ \
	(Z,\eta)  \in\mathfrak{M}_{w_{1,k}^{\bot}}^{\mathrm{a}}\cap \mathfrak{M}_{w_{2,k}^{\bot}}^{\mathrm{a}}, \ \ \
	(W_2 ,\psi_2) \in \mathfrak{M}_{w_{2,k}^{\bot}}^{\mathrm{a}}\cap \mathfrak{M}_{l_2^{\bot}}^{\mathrm{a}}.
\end{equation*}
The proposition now follows from Proposition \ref{prop:universalPHK}.
\end{proof}

\subsection{} Finally, Proposition \ref{prop:firstPart} implies Theorem \ref{thm:connectingHK} via the next lemma.

\begin{lemma}
	Let $h_1, h_2$ be linearly independent positive classes in $\Lambda$. Then there exists finitely many vectors~$l_1, l_2, \dots , l_k\in \Lambda$ such that:
	\begin{itemize}
		\item $l_1=h_1$ and $l_k=h_2$;
		\item $(l_i, l_i)>0$, for each $i=1,\dots, k$;
		\item $\langle l_i, l_{i+1}\rangle$ has signature $(1,1)$ and $(l_i, l_{i+1})>0$, for each $i=1,\dots ,k-1$.
	\end{itemize}
\end{lemma}
\begin{proof}
	The argument presented here is due to Soldatenkov~\cite[\S6.3]{soldatenkov19}.
	If $\langle h_1, h_2\rangle$ is of signature $(1,1)$ and $(h_1, h_2)>0$ there is nothing to do. 
	
	Assume that $\langle h_1, h_2\rangle$ is positive definite. 
	We may assume that $(h_1, h_2)=0$, for, otherwise, we pick a positive class $h_3\in \langle h_1, h_2\rangle^{\bot}$ and apply the argument below to~$h_1, h_3$ and $h_3,h_2$.
	Consider $V$ consisting of $l \in \Lambda\otimes \RR$ such that:
	\begin{itemize}
		\item  $(l,l)>0$, $(h_1, l)>0$ and $(h_2,l)>0$;		
		\item $\langle h_1, l\rangle$ and $\langle l, h_2\rangle$ are both of signature~$(1,1)$.
	\end{itemize} 
	Then $V$ is an open cone in $\Lambda\otimes\RR$ and it suffices to show that $V$ is not empty. Let~$w\in\langle h_1, h_2 \rangle^{\bot}$ be such that $(w,w)<0$ and let~$u_1, u_2, u_3$ be the orthogonal basis of $\langle h_1, h_2, w\rangle\otimes \RR$ such that 
	$(u_1, u_1)=1$, $ (u_2, u_2)=1$, $(u_3, u_3)=-1$, and
	$h_1=\alpha u_1$, $h_2=\beta u_2$, $w=\gamma u_3$, for positive real numbers $\alpha$, $\beta$ and $\gamma$.
	A computation shows that the vector $l=u_1+ u_2+\delta u_3$ is positive for~$\delta^2<2$, and both the real vector spaces $\langle h_1, l\rangle$ and $\langle h_2, l \rangle$ are of signature~$(1,1)$ for $\delta^2>1$. Moreover, $(h_1, l)=\alpha$ and~$(h_2, l)=\beta$ are positive. Hence, if $1<\delta^2<2$, the vector $l\in V$. 
	
	If $\langle h_1, h_2\rangle$ is of signature $(1,1)$ and $(h_1, h_2)<0$, we simply let~$l$ be a positive class in $\langle h_1, h_2 \rangle^{\bot}$. Then $\langle h_1, l\rangle$ and $\langle h_2, l\rangle$ are positive definite, and we conclude as above.
	
	Finally, let $\langle h_1, h_2 \rangle$ be degenerate.  
	The set $V$ of $l\in\Lambda\otimes \RR$ such that 
	\begin{itemize}
		\item $(l,l)>0$, and
		\item $\langle h_1, l\rangle $ and $\langle l, h_2 \rangle$ are both non-degenerate,
	\end{itemize} 
	is a non-empty open cone in $\Lambda\otimes \RR$. Hence, there exists a positive class $l\in \Lambda$ such that $\langle h_1, l\rangle$ and $\langle l, h_2 \rangle$ are both non-degenerate. This concludes the proof.
\end{proof}

\subsection{} \label{subsec:technical}
Below are proven the statements used in the proof of Proposition~\ref{prop:firstPart}.

\begin{proof}[Proof of Claim \ref{claim:lemma1}]
	By construction, the intersection matrices of $L_1$ and $L_2$ are
	\[
	\begin{pmatrix}
		p^{j}f_1 & 0 \\
		0 & p^{r_1}\epsilon_1
	\end{pmatrix} \ \ \ \text{and} \ \ \
	\begin{pmatrix}
		p^{j}f_2 & 0 \\
		0 & p^{r_2}\epsilon_2
	\end{pmatrix},
	\]
	respectively, with $j\gg 0$ odd and $f_1, f_2, \epsilon_1, \epsilon_2$ all non-zero squares modulo $p$. In particular, $\epsilon_1f_1$ and~$\epsilon_2f_2$ are squares modulo $p$.
	
	Let $v\in L_1\otimes \QQ$; the case of $v\in L_2\otimes \QQ$ is analogous. Assume that $(v,v)$ is an integer. There exist integers $\gamma$, $\lambda$, $\delta$ such that $\gamma v =\lambda l_1 + \delta v_1$. Assume that $j\geq r_1$ (otherwise we proceed similarly). 
	We then have 
	\begin{align*}
		\gamma^2 (v,v) & = p^{r_1}( \lambda^2 p^{j-r_1}f_1 + \delta^2\epsilon_1).	
	\end{align*}
	
	Assume by contradiction that $(v,v)$ is not divisible by $p$, and let $m$ be the biggest integer such that $p^m $ divides both $\lambda$ and $\delta$. 
	We can then write 
	\[
	\gamma^2 (v,v)=p^{2m + r_1}(\lambda_0^2 p^{j - r_1}f_1 + \delta_0^2\epsilon_1),
	\]
	where $p$ does not divide both $\lambda_0$ and $\delta_0$.
	The left hand-side is divisible by an even power of $p$. Since $r_1$ is odd, if $j-r_1>0$ this forces $\delta_0$ to be divisible by $p$, and hence $\delta_0= p\delta_1$ for some integer~$\delta_1$. Therefore, $\lambda_0$ is not divisible by $p$.
	We obtain
	\[
	\gamma^2 (v,v)=p^{2m + r_1 + 2}(\lambda_0^2 p^{j-r_1-2}f_1 + \delta_1^2\epsilon_1).
	\]
	If $j-r_1-2>0$, we again find that $p$ has to divide $ \delta_1$, so $\delta_0= p^2\delta_2$ and 
	\[
	\gamma^2 (v,v)=p^{2m+r_1+4}(\lambda_0^2 p^{j-r_1-4}f_1 + \delta_2^2\epsilon_1).
	\] 
	Proceeding in this way we eventually find $\bar{\delta}$ such that $\delta_0=p^{(j-r_1)/2}\bar{\delta}$ and
	\[
	\gamma^2 (v,v)=p^{2m + j}(\lambda_0^2 f_1 + \bar{\delta}^2\epsilon_1).
	\]
	Now, since $j$ is odd, $p$ has to divide $\lambda_0^2 f_1 + \bar{\delta}^2\epsilon_1$. But $-f_1\epsilon_1$ is not a square modulo $p$ because $p\equiv 3$ modulo $4$, and therefore $\lambda_0^2 f_1 + \bar{\delta}^2\epsilon_1\equiv 0$ modulo $p$ has no non-trivial solutions. 
	It follows that $p$ divides $\lambda_0$, a contradiction.
\end{proof}

\begin{proof}[Proof of Claim \ref{claim:lemma2}]
	We let $b\coloneqq (l_1, l_2)$ and $s\coloneqq (l_1, v_2)$. The intersection matrix of $S_k$ is 
	\[
	\begin{pmatrix}
		p^{2k+j}f_1+p^{r_1}\epsilon_1 & p^{2k}b + p^k s \\
		p^{2k}b +p^{k}s   & p^{2k+j}f_2 + p^{r_2}\epsilon_2
	\end{pmatrix}.
	\]
	
	Let $v\in S_k\otimes\QQ$ be such that $(v,v)\in \ZZ$. We have $\gamma v=\lambda w_{1,k} + \delta w_{2,k}$ for integers $\gamma, \lambda, \delta$. Without loss of generality, we may assume that $r_1\leq r_2$.~Then:
	\begin{align*}
		\gamma^2 (v,v) = p^{r_1} \Bigl(\lambda^2\epsilon_1+ \delta^2p^{r_2-r_1} \epsilon_2 + D\Bigr), 
	\end{align*}
where 
$$ D = \lambda^2 p^{2k+j-r_1}f_1 + 2\lambda\delta (p^{2k-r_1} b+ p^{k-r_1} s) +\delta^2p^{2k+j-r_1}f_2$$
	is divisible by a large power of $p$, since $k\gg 0$. 
	Let $m$ be such that $\lambda=p^m\lambda_0$ and $\delta=p^m\delta_0$, with at least one among $\lambda_0$ and $\delta_0$ not divisible by $p$. We can then write
	\[
	\gamma^2(v,v)  = p^{2m+r_1}(\lambda_0^2 \epsilon_1 + \delta_0^2 p^{r_2-r_1} \epsilon_2 + D_0).
	\]
	
	Assume by contradiction that $p$ does not divide $(v,v)$. Then $\gamma^2(v,v)$ is divisible by an even power of $p$, and hence, being $r_1$ odd, $p$ necessarily divides the term in brackets on the right hand side. 
    If $r_2>r_1$ this forces $\lambda_0$ to be divisible by $p$, so $\lambda_0=p\lambda_1$. 
    Proceeding as in the previous proof, we obtain $\bar{\lambda}$ such that $\lambda_0=p^{(r_2-r_1)/2}\bar{\lambda}$ and
    \[
    \gamma^2(v,v)  = p^{2m+r_2}(\bar{\lambda}^2 \epsilon_1 + \delta_0^2 \epsilon_2 + D_0/{p^{r_2-r_1}}),
    \]
     
    Note that $D_0/p^{r_2-r_1}$ is still an integer divisible by $p$. But then we must have 
	$$\bar{\lambda}^2 \epsilon_1 + \delta_0^2 \epsilon_2\equiv 0 \ \text{modulo} \ p.$$ Since $\epsilon_1\epsilon_2$ is a square modulo $p$ and $p\equiv 3$ modulo $4$, this equation has no non-trivial solutions; hence $\bar{\lambda}\equiv \delta_0\equiv 0 $ modulo $p$, which is a contradiction.
\end{proof}

\bibliographystyle{smfalpha}
\bibliography{bibliography}{}

\end{document}